\documentclass[12pt]{article}
\usepackage[english]{babel}
\usepackage{graphicx}
\usepackage{framed}
\usepackage[normalem]{ulem}
\usepackage{amsmath}
\usepackage{amsthm}
\usepackage[utf8]{inputenc}
\usepackage[english]{babel}
\usepackage{amssymb}
\usepackage{amsfonts}
\usepackage{tikz-cd}
\usepackage{enumerate}
\usepackage[utf8]{inputenc}
\usepackage{tikz}
\usepackage[top=1 in,bottom=1in, left=1 in, right=1 in]{geometry}
\usepackage{enumitem}

\usepackage{dynkin-diagrams}
\usetikzlibrary{decorations.markings}

\usepackage{hyperref}

\newcommand{\p}{\mathbb{P}}
\newcommand{\N}{\mathbb{N}}
\newcommand{\F}{\mathbb{F}}

\newcommand{\A}{\mathcal{A}}
\newcommand{\B}{\mathcal{B}}
\newcommand{\C}{\mathbb{C}}

\newcommand{\s}{\mathcal{O}}
\newcommand{\Q}{\mathbb{Q}}
\newcommand{\Z}{\mathbb{Z}}
\newcommand{\<}{\langle}
\renewcommand{\>}{\rangle}
\renewcommand{\emptyset}{\varnothing}

\theoremstyle{plain}
\newtheorem{theorem}{Theorem}[section] 

\theoremstyle{definition}
\newtheorem{prop}[theorem]{Proposition}
\newtheorem{lemma}[theorem]{Lemma}
\newtheorem{conj}[theorem]{Conjecture}

\newtheorem{Not}[theorem]{Notation}

\newtheorem{Remark}[theorem]{Remark}

\newtheorem{cor}[theorem]{Corollary}

\newtheorem{defn}[theorem]{Definition} 
\newtheorem{exmp}[theorem]{Example}

\usepackage{mathtools}
\usepackage{multirow}

\tikzstyle{vertex} = [fill,shape=circle,node distance=80pt]
\tikzstyle{edge} = [fill,opacity=.5,fill opacity=.5,line cap=round, line join=round, line width=50pt]
\tikzstyle{elabel} =  [fill,shape=circle,node distance=30pt]

\title{Relative Kazhdan-Lusztig isomorphism for $GL_{2n}/Sp_{2n}$}
\author{Guy Shtotland \\
\small Department of Mathematics\\
\small Ben-Gurion University of the Negev}
\date{}

\begin{document}

\maketitle
\setlength\parindent{0pt}

\begin{abstract}
The Kazhdan–Lusztig isomorphism, relating the affine Hecke algebra of a $p$-adic group to the equivariant K-theory of the Steinberg variety of its Langlands dual, played a key role in the proof of the Deligne–Langlands conjectures concerning the classification of smooth irreducible representations with an Iwahori fixed vector. In this work we state and prove a relative version of the Kazhdan–Lusztig isomorphism for the symmetric pair $(GL_{2n},Sp_{2n})$. The relative isomorphism is an isomorphism between the module of compactly supported Iwahori invariant functions on $X=GL_{2n}/Sp_{2n}$ and another module over the affine Hecke algebra constructed using equivariant $K$-theory and the relative Langlands duality. We use this isomorphism to give a new proof of a condition on $X$ distinguished representations.  

\end{abstract}

\begin{section}{Introduction}
Let $F$ be a non-archimedean local field, $\s$ its ring of integers and $k$ its residue field. Denote $q_r=\#k$.

Let $\mathbf{G}$ be a reductive connected group split over $F$ and let $G=\mathbf{G}(F)$ be its $F$ points.

Let $B$ be a Borel subgroup of $G$ and let $I$ be an Iwahori subgroup of $G$. Let $H(G,I)$ be the affine Hecke algebra of $G$, it is the algebra of $I$ bi-invariant compactly supported functions on $G$. 

Let $H_q$ be the generic version of $H(G,I)$, it depends on a complex parameter $q\in\C$. For $q=q_r$ we get $H_{q_r}\cong H(G,I)$.

Let $G^\vee$ be the complex dual group of $G$. Let $\mathfrak{g}^\vee$ be the Lie algebra of $G^\vee$ and let $\mathfrak{g}^{*\vee}$ be its dual. Denote by $Ad$ the natural action of $G^\vee$ on $\mathfrak{g}^{\vee}$. Let $\B$ be the flag variety of $G^\vee$. For $B^\vee\in\B$ we denote by $\mathfrak{b}^\vee$ its Lie algebra.

The Deligne Langlands conjecture gives a classification of the irreducible representations of $G$ with an $I$ fixed vector, in terms of $G^\vee$. 

\begin{theorem}{\cite{Kazhdan1987ProofOT},\cite{reeder}} \label{DL conjecture}The irreducible smooth representations of $G$ with an $I$ fixed vector are parametrized by conjugacy classes of triples $(t,n,\chi)$ with $t\in G^\vee$ semi simple, $n\in \mathfrak{g}^{\vee}$ such that $Ad_t(n)=q_rn$. The element $\chi$ is an irreducible representation of the finite group $C_{G^\vee}(t,n)/C^0_{G^\vee}(t,n)$ that appears in the representation $H_\bullet(\B_n^t)$. Here, $C_{G^\vee}(t,n)$ is the common centralizer of $t,n$ and $C^0_{G^\vee}(t,n)$ is the connected component of the identity in $C_{G^\vee}(t,n)$. The variety $\B_n^t$ is the variety of Borel subgroups $B^\vee\in\B$ such that $n\in \mathfrak{b}^\vee$ and $t\in B^\vee$.

We refer to $(t,n,\chi)$ as the Deligne Langlands parameter of $\pi$.
    
\end{theorem}

A key step in the proof of The Deligne Langlands conjecture is a geometric description of $H_q$ using equivariant $K$ theory.

 Let $\Tilde{N}=T^*\B$ be the cotangent bundle of the flag variety of $G^\vee$. It can be described as  $\Tilde{N}=\{(B^\vee,\phi)|B^\vee\in\B,\phi\in \mathfrak{g}^{\vee*}, \phi|_{\mathfrak{b}^\vee}=0\}$.
 
 Let $St=\Tilde{N}\times_{\mathfrak{g}^{\vee*}}\Tilde{N}=\{(B^\vee_1,B^\vee_2,\phi)|B^\vee_1,B^\vee_2\in\B, \phi\in \mathfrak{g}^{\vee*}, \phi|_{\mathfrak{b}_1^\vee}=\phi|_{\mathfrak{b}_2^\vee}=0\}$ be the Steinberg variety of $G^\vee$. The group $G^\vee\times\C^\times$ acts on $St$, $G^\vee$ acts in the obvious way and $\C^\times$ acts on $\mathfrak{g}^{\vee*}$ by $z,\phi\mapsto z^2\phi$.

    The three projection maps from $\Tilde{N}\times_{\mathfrak{g}^{\vee*}}\Tilde{N}\times_{\mathfrak{g}^{\vee*}}\Tilde{N}$ to $St$, give us a convolution product on homology theories of $St$.

    Let $W$ be the Weyl group of $G$ and let $W_{aff}$ be the extended affine Weyl group of $G$.

The proofs of the following results can be found in \cite{Kazhdan1987ProofOT} and  \cite{Chriss1997RepresentationTA} .

\begin{theorem}\label{classic}
    \begin{enumerate}
     \item The number of irreducible components of $St$ is equal to the size of the Weyl group $W$.
    
        \item For top Borel Moore homology of $St$, there is an isomorphism $\Phi_f:H^{BM}_{top}(St)\xrightarrow{\sim}\mathbb{C}[W]$. 
        \item For equivariant K theory, there is an isomorphism  $\Phi_a:K^{G^\vee}(St)\xrightarrow{\sim} \mathbb{C}[W_{aff}]$.
        \item For equivariant K theory, there is an isomorphism $\Phi_{KL}:K^{G^\vee\times \mathbb{C}^\times}(St)\xrightarrow{\sim} H_q$. 
    \end{enumerate}
\end{theorem}

We prove a relative version of Theorem \ref{classic} for the symmetric space $X=GL_{2n}/Sp_{2n}$.

Let $C^\infty(X)$ be the space of locally constant functions on $X$. Let $S(X)$ be the space of compactly supported locally constant functions on $X$, and let $S(X)^I$ be the space of $I$ invariant compactly supported functions on $X$. The algebra $H(G,I)$ acts on $S(X)^I$ by convolution.

In \cite{benzvi2024relativelanglandsduality}, a generalization of Langlands duality is suggested. For every spherical variety $X$ with a $G$ action, we can consider $M=T^*X$, a Hamiltonian space. Under some conditions, a dual Hamiltonian space $M^\vee$ with a $G^\vee$ action is defined. It comes with a $\mathbb{G}_m$ action that commutes with the $G^\vee$ action.

For the $GL_{2n}$ space $X=GL_{2n}/Sp_{2n}$ the attached dual space $M^\vee$ is described as follows.
Let $H\subset GL_{2n}$ be the group of block matrices of the form $H=\{\begin{pmatrix}
        g & a  \\
        0 & g
        \end{pmatrix}|g\in GL_n,a\in M_n\}$. Denote by $\mathfrak{h}$ the Lie algebra of $H$ and by $\mathfrak{h}^*$ its dual. We fix $\psi\in (\mathfrak{h}^*)^H$, an $H$ invariant element of $\mathfrak{h}^*$, $\psi(\begin{pmatrix}
        g & a  \\
        0 & g
        \end{pmatrix})=tr(a)$. 

        Let $M^\vee=T_\psi^*(GL_{2n}/H)$ be the twisted cotangent bundle attached to $\psi$.

        It can be described explicitly as $M^\vee=\{(gH,\phi),gH\in G^\vee/H, \phi\in \mathfrak{g}^*, (g^{-1}\phi)|_\mathfrak{h}=\psi\}$.
         
         By table 1.5.1 of \cite{benzvi2024relativelanglandsduality} this is the dual space to $T^*(GL_{2n}/Sp_{2n})$. 

       Recall that the Weyl group $W$ acts on $B\backslash X$ (see \cite{knop}) and that the extended affine Weyl group $W_{aff}$ acts on $I\backslash X$ (see \cite{my}).

        We consider the space $\Lambda=M^\vee\times_{\mathfrak{g}^*}\Tilde{N}$, this is a relative analogue of $St$. There are three projection maps from $M^\vee\times_{\mathfrak{g}^*}\Tilde{N}\times_{\mathfrak{g}^*}\Tilde{N}$, two to $\Lambda$ and one to $St$. These maps give a module structure on homology theories of $\Lambda$ over homology theories of $St$ as in Subsection 5.2.20 of \cite{Chriss1997RepresentationTA}.

        Let $sgn$ be the sign representation of $W$, and let $sgn_f$ be its extension to $W_{aff}$.
        
        Let $IM:H_q\rightarrow H_q$ be the Iwahori Matsumoto involution, it induces an involution on $H_q$ modules. For each $H_q$ module $V$ we denote by $IM(V)$ the $H_q$ module obtained from $V$ by twisting the $H_q$ action by this involution. That is, $h\in H(G,I)$ acts on $IM(V)$ the same way $IM(h)$ acts on $V$. 

        Let $\C[B\backslash X]$ be the vector space spanned by $B\backslash X$. Similarly, let $\C[I\backslash X]$ be the vector space spanned by $I\backslash X$. 

        Let $M_q$ be the generic version of the module $S(X)^I $ as defined in \cite{my}.
        
        We prove the following results.

        \begin{theorem}\label{relative KL}
            \begin{enumerate}
    \item The number of Borel orbits on $X$ is equal to the number of irreducible components of $\Lambda$.
    
    \item There is a module isomorphism $\Phi_{f,X}:H^{BM}_{top}(\Lambda) \otimes sgn\xrightarrow{\sim} \mathbb{C}[B\backslash X]$ compatible with $\Phi_f$.
    
    \item There is a module isomorphism $\Phi_{a,X}:K^{G^\vee}(\Lambda) \otimes sgn_f\xrightarrow{\sim} \C[I\backslash X]$ compatible with $\Phi_a$.

    \item 
    
    There is a module isomorphism $\Phi_{KL,X}:IM(K^{G^\vee\times\mathbb{C}^\times}(\Lambda))\xrightarrow{\sim} M_q$ compatible with $\Phi_{KL}$. 
\end{enumerate}
        \end{theorem}

\begin{Remark}
    The first part of the above theorem is a special case of a conjecture made in \cite{finkelberg2023lagrangiansubvarietieshypersphericalvarieties}.

\end{Remark}

\begin{Remark}
    The paper \cite{finkelberg2023lagrangiansubvarietieshypersphericalvarieties} also contains a conjecture related to the second part. There, it is proposed that $H^{BM}_{top}(\Lambda)\cong H^{BM}_{top}(M\times_{\mathfrak{g}^{*}} T^*(G/B))$ as $W$ representations. This conjecture is not precise even in the case of the Whittaker model which is dual to a point. 
    
    For this example we get $H^{BM}_{top}(T^*_\psi(G/U)\times_{\mathfrak{g}^{*}} T^*(G/B))=sgn$ and $H^{BM}_{top}(pt\times_{\mathfrak{g}^{\vee*}} T^*(G^\vee/B^\vee))$ is trivial as a $W$ representation. In particular, the two representations are not isomorphic (for any reductive connected group $G$), they differ by a sign.

    Following the example of the Whitaker model, we propose that the more correct conjecture is that $H^{BM}_{top}(\Lambda)\cong H^{BM}_{top}(M\times_{\mathfrak{g}^{*}} T^*(G/B))\otimes sgn$. We prove this for our example.

    \end{Remark}

We use Theorem \ref{relative KL} to prove a result about irreducible $GL_{2n}$ representations with an $I$ fixed vector that are $X$ distinguished.

\begin{defn}\label{dist rep}
    An irreducible representation $\pi$ of $G$ is called $X$ distinguished if 
    
    $Hom_G(\pi,C^\infty(X))\neq 0$. 
\end{defn}

 We recover a result of \cite{MitraOffenSayag2017} that gives a condition on $X$ distinguished representations. In order to state this condition we pass to the language of Zelevinsky parameters (see \cite{Zelevinsky1980}).

\begin{defn}\label{multisegments}
 A segment is a finite set of numbers $\{a_1,...,a_l\}$ such that $a_1<a_2<...,<a_l$ and $a_t-a_{t-1}=1$ for $1<t\leq l$. A multi-segment is a multi-set of segments. Zelevinsky proved that irreducible representations of $G$ with an $I$ fixed vector are parametrized by multi-segments (see \cite{Zelevinsky1980}). This is closely related to the parametrization given by Theorem \ref{DL conjecture}.  
    
\end{defn}

\begin{theorem}\label{even condition}
    Let $\pi$ be an irreducible representation of $G$ that is $X$ distinguished, then all the segments in the Zelevisnky parameter of $\pi$ have even length.
\end{theorem}

\begin{Remark}
    As mentioned, this was first proven in \cite{MitraOffenSayag2017}. Our proof is completely different.
\end{Remark}

\begin{Remark}
 This does not solve the classification problem of $X$ distinguished representations with an $I$ fixed vector because not all representations which satisfy the aforementioned condition are $X$ distinguished. We do not have a general solution to this problem. Nevertheless, we give some examples where we do manage to solve this problem. We also give some conditions which may lead to a solution.
   
\end{Remark}

\begin{subsection}{General conjectures}
In this subsection we propose a conjecture generalizing Theorem \ref{relative KL} as well as a conjecture generalizing Theorem \ref{even condition}. 

Let $\sigma:G\rightarrow G$ be an algebraic involution and let $H=G^\sigma$ be its fixed points. Let $X=G/H$ be a symmetric space. Assume that the characteristic of $k$ is not $2.$

\begin{defn}
Let $M_q$ be the module attached to $X$ introduced in \cite{my}. We define a quotient of this module, which we denote by $\Tilde{M}_q$, by identifying $I$ orbits which become the same over a separable closure of $F$. The action of the affine Hecke algebra $H_q$ descends to an action on $\Tilde{M}_q$

A similar construction to what was done in \cite{my} can be done also for Borel orbits over a finite field instead of Iwahori orbits over a non archimedean local field. We denote the module obtained from this construction by $M^f_q$, it is a module over the generic Hecke algebra of the finite Weyl group $W$, we denote this algebra by $H^f_q$. We can again take a quotient by identifying orbits that become the same over an algebraic closure. The resulting module is denoted by $\Tilde{M}^f_q$.

\end{defn}

We make two assumptions:
\begin{enumerate}
    \item the Hamiltonian space $M=T^*X$ is hypershperical with hyperspherical dual $M^\vee$ (see \cite{benzvi2024relativelanglandsduality}). 
   
    \item Both $G$ and $H$ are split over $F$.
\end{enumerate}

\begin{Remark}
    One can also consider the assumption $1'$ below , it is related to assumption $1$ above.
    \begin{enumerate}[label=\arabic*$'$.]
         \item The symmetric space $X$ has no roots of type $N$ (see 2.1.4 of \cite{Ressayre_2010} for the definition of a root of type $N$).
    \end{enumerate}
    Under the assumptions $1'$ and $2$, the module $\Tilde{M}_q$ is isomorphic to the affine Lusztig-Vogan module defined in \cite{chen2025singularitiesorbitclosuresloop}. 

    Under the same assumptions, the module $\Tilde{M}^f_q$ is isomorphic to the Lusztig-Vogan module constructed in \cite{Lusztig1983}.
\end{Remark}

 We have a moment map $\mu:M^\vee\rightarrow \mathfrak{g}^{\vee*}$. We denote $\Lambda=M^\vee\times_{\mathfrak{g}^{\vee*}}\Tilde{N}$, this variety has an action of $G^\vee\times\C^\times$.

We make the following conjecture.

\begin{conj}
\begin{enumerate}
        \item The number of Borel orbits on $X$ over an algebraically closed field is equal to the number of irreducible components of $\Lambda$.
    
    \item There is an isomorphism $\Phi_{f,X}:H^{BM}_{top}(\Lambda) \otimes sgn\xrightarrow{\sim} \Tilde{M}^f_1$ compatible with $\Phi_f$.
    
    \item There is an isomorphism $\Phi_{a,X}:K^{G^\vee}(\Lambda) \otimes sgn_f\xrightarrow{\sim} \Tilde{M}_1$ compatible with $\Phi_a$.

    \item 
    
    There is an isomorphism $\Phi_{KL,X}:IM(K^{G^\vee\times\mathbb{C}^\times}(\Lambda))\xrightarrow{\sim} \Tilde{M}_q$ compatible with $\Phi_{KL}$. 
\end{enumerate}
\end{conj}

We also propose a conjecture about distinguished irreducible representations with an $I$ fixed vector.

For a smooth $G$ representation $\pi$ we denote by $Z(\pi)$ the Aubert-Zelevinsky dual representation (see \cite{Aubert}). Let $Z(\pi)^\vee$ be the contragridient, i.e. smooth dual, of $Z(\pi)$.

We use a Killing form on $\mathfrak{g}^\vee$ to identify $\mathfrak{g}^\vee\cong \mathfrak{g}^{\vee*}$. We can consider an element $n\in \mathfrak{g}^\vee$ as $n\in\mathfrak{g}^{\vee*}$.

\begin{conj}
    Let $\pi$ be an irreducible representation of $G$ with $\pi^I\neq 0$. Let $(t,n,\chi)$ be the Deligne Langlands parameter of $\pi$. Let $a=(t,\sqrt{q_r})\in G^\vee\times\C^\times$. If $Z(\pi)^\vee$ is $X$ distinguished then $(M^\vee)^a\neq 0$ and $n\in (\mathfrak{g}^{*\vee})^a$ is in the image of the moment map $\mu:(M^\vee)^a\rightarrow (\mathfrak{g}^{*\vee})^a$. 
\end{conj}

\end{subsection}

\begin{subsection}{Methods of proof}

We use $H(G,I)$ and $S(X)^I$ also to denote their generic version $H_q$ and $M_q$. We use $q$ to denote a formal variable.

We consider the projection $\tau:\Lambda\rightarrow \B\times G^\vee/H$. There is a natural bijection between the irreducible components of $\Lambda$ and the $G^\vee$ orbits on $Im(\tau)$ (see for example \cite{kononenko2024lagrangiansubvarietieshypersphericalvarieties}). We construct a a map $\Phi_X:B\backslash X\rightarrow G^\vee\backslash Im(\tau)$ and prove the following.

\begin{prop}
    The map $\Phi_X$ is a bijection and it reverses the weak Bruhat order on $G^\vee\backslash Im(\tau)$ and on $B\backslash X$. 
\end{prop}

This of course implies the first part of Theorem \ref{relative KL}.

 Now we describe the argument to prove part two of Theorem \ref{relative KL}. We denote by $\mathcal{U}_{max}\in B\backslash X$ the open orbit, we use the same notation also for the element of $\C[B\backslash X]$. We have $\Lambda_0$ the irreducible component of $\Lambda$ corresponding to $\Phi_X(\mathcal{U}_{max})$. We show that $\mathcal{U}_{max}$ and $[\Lambda_0]$ generate $\C[B\backslash X]$ and $H_{top}^{BM}(\Lambda)$ respectively, as modules over $\C[W]$. We then show that the map $\C[W]\rightarrow H_{top}^{BM}(\Lambda)\otimes sgn$ given by action on $[\Lambda_0]$ factor through the map $\C[W]\rightarrow\C[B\backslash X]$ given by acting on $\mathcal{U}_{max}$. This gives us a map $\C[B\backslash X]\rightarrow H_{top}^{BM}(\Lambda)\otimes sgn$ and we prove it is an isomorphism.

 The proof of part three combines arguments similar to the argument for part two and a cellular fibration argument.  Let $S=\{\phi\in \mathfrak{g}^{\vee*},\phi|_\mathfrak{h}=\psi\}$, $H$ acts on $S$. We have $\Lambda=G\times^H\Lambda_H$ for the $H$ space $\Lambda_H=\Tilde{N}\times_{\mathfrak{g}^{\vee*}}S$. We show that $\Lambda_H$ satisfies the cellular fibration Lemma (see Lemma \ref{cfl}) over the flag variety of $GL_n$. This allows us to describe $K^{G^\vee}(\Lambda)\cong K^H(\Lambda_H)$.

 The proof of part four requires all the arguments used in order to prove part three as well as an additional one. Let $w\in\mathbf{G}(\s)$ be such that $BwSp_{2n}=\mathcal{U}_{max}$. We show that the map $H(G,I)\rightarrow IM(K^{G^\vee\times \C^\times}(\Lambda))$ given by acting on the structure sheaf of $\Lambda_0$, factors thorough the map $H(G,I)\rightarrow S(X)^I$ given by acting on the characteristic function $1_{IwSp_{2n}}$. In order to show this we compute the annihilator of $1_{IwSp_{2n}}\in S(X)^I$. To compute this annihilator we use the most degenerate boundary degeneration of $X$, denoted by $X_\emptyset$ and the Bernstein morphism $e:S(X_\emptyset)^I\rightarrow S(X)^I$. In Appendix \ref{A1} we present some results about boundary degenerations of symmetric spaces which are interesting in their own right.

 Let us also outline the main ingredients of our proof of Theorem \ref{even condition}.

 For $a\in G^\vee\times\C^\times$ and a variety $Y$ with a $G^\vee\times\C^\times$ action, we denote by $Y^a$ the fixed points of $a$.

 Let $\chi$ be a central character of $H(G,I)$, i.e. a character of the center $Z=Z(H(G,I))$.
The algebra $Z$ is isomorphic to the representation ring of the group $G^\vee\times\C^\times$. We can associate to $\chi$ a semi-simple conjugacy class of $G^\vee$. Let $t$ be a semi-simple element of $G^\vee$ which is a representative of such a class. Let $v$ be the positive square root of $q_r$ and let $a=(t,v)\in G^\vee\times\C^\times$.

Using standard arguments taken from Section 5 of \cite{Chriss1997RepresentationTA} and the cellular fibration lemma we show the following.

\begin{prop}\label{1.16}
 We have an isomorphism $H^{BM}_\bullet(\Lambda_H^a)\cong IM(S(X)^I)\otimes_{Z(H(G,I))} \C_a$ as $H^{BM}_\bullet(St^a)\cong H(G,I)\otimes_{Z(H(G,I))} \C_a$ modules.
\end{prop}

We use a perverse sheaf description of $H^{BM}_\bullet(St^a)$ and $H^{BM}_\bullet(\Lambda_H^a)$ (as in Subsection 8.6 of \cite{Chriss1997RepresentationTA}) to prove Theorem \ref{even condition}.

Let $C_{G^\vee}(t)$ be the centralizer of $t$ and let $\mathcal{E}$ be the set of $C_{G^\vee}(t)$ orbits on $(\mathfrak{g}^{\vee *})^a$. For each $c\in\mathcal{E}$ we denote by $IC_{c}$ the intersection cohomology sheaf of the orbit $c$ (with the trivial local system on it). Recall $S=\{\phi\in \mathfrak{g}^{\vee*},\phi|_\mathfrak{h}=\psi\}\subset \mathfrak{g}^{\vee*}$. Let $\mu:S\hookrightarrow \mathfrak{g}^{\vee*}$ be the embedding. Using Proposition \ref{1.16} we prove the following condition.

\begin{prop}\label{short sheaf condition}
    Let $\pi$ be an irreducible representation of $G$, generated by its $I$ fixed vectors. Let $(t,n)$ be a Deligne Langlands parameter of $\pi$. Using a Killing form one can identify $\mathfrak{g^\vee}\cong\mathfrak{g}^{\vee*}$ and consider $n\in \mathfrak{g}^{\vee*}$. Let $c\in\mathcal{E}$ be the orbit of $n$.

    The following are equivalent:

    \begin{enumerate}
        \item Let $Z(\pi)$ be the Zelevinsky dual of $\pi$ (see Section 9 of \cite{Zelevinsky1980}) and let $Z(\pi)^\vee$ be its contragrident. The representation $Z(\pi)^\vee$ is $X$ distinguished. 
        \item The map $\bigoplus_{k>0,d\in\mathcal{E}}Ext^k(IC_c,IC_d)\otimes \bigoplus Ext^{\bullet}(IC_d,\mu_*\C_{S^a})\rightarrow\bigoplus Ext^\bullet(IC_c,\mu_*\C_{S^a})$ is not surjective.
    \end{enumerate}
    
\end{prop}

We pass to the settings of etale $l$-adic sheaves in order to use the theory of weights and pure sheaves. In this setting, we show that the second condition of Proposition \ref{short sheaf condition} can not be satisfied unless $n$ is in the image of the moment map $M^\vee\rightarrow \mathfrak{g}^{\vee*}$. This implies Theorem \ref{even condition}.

\end{subsection}

\begin{subsection}{Structure of the paper}

In Section \ref{s2} we recall some classical results from geometric representation theory.

In Section \ref{s3} we prove the first part of Theorem \ref{relative KL}.

In Section \ref{s4} we prove the second part of Theorem \ref{relative KL}. We also prove a version of the Conjecture 1.1.2 of \cite{finkelberg2023lagrangiansubvarietieshypersphericalvarieties} for this case.

In Section \ref{s5} we describe the $I$ orbits on $X$ and the module $S(X)^I$ over $H(G,I)$.

In Section \ref{s6} we show that $\Lambda_H$ satisfies the cellular fibration lemma.

In Section \ref{s7} we prove the third and fourth parts of Theorem \ref{relative KL}.

In Section \ref{s8} we prove Proposition \ref{1.16} and deduce a condition on the central character of an irreducible representation $\pi$ with an $I$ fixed vector that is $X$ distinguished.

In Section \ref{s9} we give our proof of  Theorem \ref{even condition}.

\end{subsection}

\begin{subsection}{Acknowledgment}

I would like to thank Shachar Carmeli, Michael Finkelberg, Nadya Gurevitch, Guy Kapon, Toan Pham, and Eitan Sayag for helpful discussions. I was partially supported by ISF grant No. 1781/23.

\end{subsection}

\end{section}

\begin{section}{Preliminaries}\label{s2}
In this section we recall known results (from \cite{Chriss1997RepresentationTA}) that are used throughout the paper.

\begin{subsection}{The convolution construction}\label{convultion}
    We have three projections  from $M^\vee\times_{\mathfrak{g}^{\vee*}}T^*\B\times_{\mathfrak{g}^{\vee*}}T^*\B$. Two projections $\pi_1,\pi_2$ to $\Lambda$, and one projection $\pi_3$ to $St$. Notice that $\pi_2$ is proper because $\B$ is proper.

    This gives a module structure for every homology theory of $\Lambda$ over the homology theory of $St$. The examples we use are Borel Moore homology and equivariant $K$ theory. 

    The convolution action is given by $X*Y=\pi_{2*}(\pi_1^*X\otimes \pi^*_3Y)$ for $X$ on $\Lambda$ and $Y$ on $St$.
\end{subsection}

\begin{subsection}{Cellular Fibration Lemma}

We recall the formulation of the cellular fibration Lemma. It is a main tool in the proof of the last two parts of Theorem \ref{classic} (see \cite{Chriss1997RepresentationTA}) and it is also a main tool for the proof of the last two parts of Theorem \ref{relative KL} in this paper.

\begin{lemma}[The Cellular Fibration Lemma, 5.5.1 of \cite{Chriss1997RepresentationTA}]\label{cfl}

Let $\pi:F\rightarrow X$ be a morphism of $G$ varieties. Assume there is a filtration $\emptyset=F^0\subset F^1\subset ...\subset F^n=F$ such that the following conditions hold.

\begin{enumerate}
    \item $F^i$ is $G$ stable closed subvariety, $\pi:F^i\rightarrow X$ is locally trivial fibration.
    \item  The map $\pi:F^{i+1}\setminus F^i\rightarrow X$ is an affine locally trivial fibration.
    \item $K^G(X)$ is a free $R(G)$ module. Here $R(G)$ is the representation ring of $G$.
\end{enumerate}

In this case we will say that $F$ satisfies the cellular fibration lemma over $X$. When this happens we have:

\begin{enumerate}
    \item There are canonical short exact sequences:
    $$0\rightarrow K^G(F^i)\rightarrow K^G(F^{i+1})\rightarrow K^G(F^{i+1}\setminus F^i)\rightarrow 0$$
    \item These sequences non canonically split as $R(G)$ modules. Moreover, $K^G(F)$ is a free $R(G)$ module.
    \item If $A\subset G$ is an algebraic subgroup such that $K^A(X)$ is a free $R(A)$ module and the map $R(A)\otimes_{R(G)} K^G(X)\rightarrow K^A(X)$ is an isomorphism then so is the map 
    
    $R(A)\otimes_{R(G)} K^G(F)\rightarrow K^A(F)$. 
\end{enumerate}
    
\end{lemma}
    
\end{subsection}

\end{section}

\begin{section}{Borel orbits and irreducible components}\label{s3}
In this section we prove part 1 of Theorem \ref{relative KL}.

First, we give an explicit description of $\Lambda$.

Recall that we have the following.

$$M^\vee=\{(gH,\phi)|gH\in G^\vee/H,\phi\in \mathfrak{g}^{\vee*},(g^{-1}\phi)|_\mathfrak{h}=\psi\}$$

The moment map $\mu:M^\vee\rightarrow \mathfrak{g}^{\vee*}$ is just the projection on the second coordinate.

We fix some Borel $B^\vee$ of $G^\vee$ with Lie algebra $\mathfrak{b}^\vee$, we have:  

$$\Lambda=M^\vee\times_{\mathfrak{g}^{\vee*}}T^*\B=\{(g_1H,g_2B^\vee,\phi)|g_1H\in G^\vee/H,g_2B^\vee\in \B,\phi\in \mathfrak{g}^{\vee*},(g^{-1}_1\phi)|_\mathfrak{h}=\psi,(g^{-1}_2\phi)|_{\mathfrak{b}^\vee}=0\}$$

We denote by $\tau$ the projection $\tau:\Lambda\rightarrow G^\vee/H\times \B$ on the first two coordinates, and by $\mu$ the projection on the third one.

\begin{defn}
    We call a $G^\vee$ orbit on $G^\vee/H\times \B$ relevant if it is in the image of $\tau$. We freely use this term also for $H$ orbits on $\B$ and for $B^\vee$ orbits on $G^\vee/H$.

    This definition agrees with the one given in \cite{kononenko2024lagrangiansubvarietieshypersphericalvarieties}.
\end{defn}

It is proven in \cite{kononenko2024lagrangiansubvarietieshypersphericalvarieties} that the number of irreducible components of $\Lambda$ is equal to the number of relevant orbits. The irreducible components are given by the closers of the preimages of the relevant orbits under $\tau$.

We describe the relevant orbits. 

Denote by $S_m$ the permutation group on $m$ elements.

\begin{prop}
    Let $W=S_{2n}$ be the Weyl group of $G^\vee=GL_{2n}$, its action on the Borel orbits of $G^\vee/H$ is transitive. There is a unique closed Borel orbit and the stabilizer of the action of $W$ on this closed orbit is $S_n$ embedded into $S_{2n}$ by the map $\alpha$.  $$\alpha(\tau)(x)=\begin{dcases}
    \tau(x),& \text{if } x\leq n\\
    \tau(x-n)+n,              & \text{otherwise}
\end{dcases}$$
\end{prop}
\begin{proof}
    Notice that $G^\vee/H$ is parabolically induced from $GL_n\times GL_n/\Delta GL_n$ in the sense of \cite{Brion2001OnOC}. Let $P$ be the parabolic of $G^\vee$ whose unipotent part is equal to the unipotent part of $H$. Write $P=LU$ where $L$ is a Levi of $P$ and $U$ is its unipotent radical. Let $W_L=S_n\times S_n$ be the Weyl group of $L$ and denote $W^L=W/W_L$.
    
    The Weyl group of $GL_n\times GL_n$ acts transitively on the Borel orbits of $GL_n\times GL_n/\Delta GL_n$. The Borel orbits are parametrized by $S_n$.

    By Lemma 7 of \cite{Brion2001OnOC} the Borel orbits on $X$ can be described as $W\times^{W_L} S_n=S_{2n}\times^{S_n\times S_n} S_n\cong S_{2n}/S_n$. Thus $S_{2n}$ acts transitively. The closed orbit is given by $(1,1)\in S_{2n}\times^{S_n\times S_n} S_n$ and its stabilizer is $S_n$ embedded as desired.
\end{proof}

Now, we want to single out the relevant orbits. Recall that $B^\vee$ is the Borel of upper triangular matrices. The orbit of $w\in W$ is relevant if and only if $\psi$ is trivial on $H\cap w^{-1}B^\vee w$. This is equivalent to all the matrices in $w^{-1}B^\vee w$ having zero $(i,i+n)$ entries. In terms of permutations this means that an element $\sigma\in S_{2n}/S_n$ is relevant if and only if $\sigma(i)>\sigma(i+n)$ for all $1\leq i\leq n$.

\begin{prop}
    The set of relevant orbits is in bijection with the set of ways to divide the numbers $1,...,2n$ to pairs.
\end{prop}

\begin{proof}
    For every $\sigma\in S_{2n}/S_n$ we assign the collection of pairs $\{\sigma(i),\sigma(i+n)\}$. Every collection of pairs is obtained uniquely from a permutation corresponding to a relevant orbit.
\end{proof}

Now we turn to the other side, meaning to $X$. $GL_{2n}$ acts transitively on the space of anti-symmetric matrices. The stabilizer being $Sp_{2n}$. Under the identification of $X$ with anti-symmetric matrices, the Borel orbits are represented by monomial matrices with ones above the diagonal. To any such matrix we attach a unique partition of $1,...,2n$ into pairs, given by the indices of the entries of the matrix which are not zero. Thus we obtain:

\begin{cor}
    The number of irreducible components of $\Lambda$ is equal to the number of Borel orbits on $GL_{2n}/Sp_{2n}$ and is equal to $(2n-1)!!$.
\end{cor}

We got a natural bijection between these sets as both sets are in bijection with the set of ways to the partition $1,...2n$ into pairs. We denote this particular bijection from  $B\backslash X$ to the relevant orbits in $B^\vee\backslash G^\vee/H$ by $\Phi_X$.

Recall that $W$ acts on the set of Borel orbits on a spherical variety.

The action of $W$ on $B^\vee\backslash G^\vee/H$ can not preserve the set of relevant orbits as it is transitive. We define a different action of $W$ on the set of relative orbits. 

\begin{defn}
    Let $w\in W$, write it as a reduce product of simple reflections $w=s_1\cdot...\cdot s_k$. Let $x$ be a Borel orbit on $G^\vee/H$, we define $s_i*x$ to be the usual product $s_ix$ if $s_ix$ is relevant and $x$ if $s_ix$ is not relevant. We define the action of $w$ to be the composition of the actions of $s_1,...,s_k$.
\end{defn}

\begin{prop}\label{group map}
\begin{enumerate}
    \item The above definition does not depend on the choice of factorization of $w$ to simple reflections.
    \item The above definition gives a product.
    \item The map $\Phi_X$ respects the $W$ action on both sides.
    
\end{enumerate} 
\end{prop}

\begin{proof}
    It is enough to show that for any simple reflection $s$ and $y\in B\backslash X$ we have $\Phi_X(sy)=s*\Phi_X(y)$. 

     We identify $y$ with an anti-symmetric matrix $m$. The orbit $sy$ corresponds to the matrix $sms$. Assume that $s$ is the reflection that switches $i$ and $i+1$. Let $\alpha,\beta\in \{1,...,2n\}$ be such that $\{i,\alpha\}$ and $\{i+1,\beta\}$ are pairs defined by $m$. On the level of pairs $s$ sends the pairs $\{i,\alpha\}$ and $\{i+1,\beta\}$ to the pairs $\{i,\beta\}$ and $\{i+1,\alpha\}$. Let $\sigma\in S_{2n}/S_n$ be some permutation corresponding to a relevant orbit and to our partition to pairs. Meaning, there are $1\leq j,k\leq n$ such that $\{\sigma(j),\sigma(j+n)\}=\{i,\alpha\}$ and $\{\sigma(k),\sigma(k+n)\}=\{i+1,\beta\}$. 
    
    The permutation $s\sigma$ is defined by $(s\sigma)(l)=\begin{dcases}
    \sigma(l),& \text{if } l\neq i,i+1\\
    \sigma(i),              & l=i+1\\
    \sigma(i+1), & l=i
\end{dcases}$

If $\alpha\neq i+1$ then $s\sigma$ is also relevant and thus $s*\Phi_X(y)=s*\sigma=s\sigma=\Phi_X(sy)$. If $\alpha=i+1$ then $s\sigma$ is not relevant and $s*\Phi_X(y)=\Phi_X(y)=\Phi_X(sy)$. 
\end{proof}

On the Borel orbits of a spherical variety we do not have just the action of $W$, we also have the weak Bruhat order. 

\begin{defn}
    Consider a spherical variety with a transitive action of $W$ on its Borel orbits.
    Let $x,y$ be two Borel orbits of this spherical variety. We say that $x<y$ If there is a series of simple reflection $s_1,...,s_k$ such that $s_1...s_kx=y$ and for each $1\leq i< k$ the closure of the orbit $s_i...s_kx$ contains the orbit $s_{i+1}...s_kx$. This order relation is called the weak Bruhat order.

    This order can be defined also without the assumption of the transitive action, we will only need this case in our paper.
\end{defn}

We can restrict the weak Bruhat order on $B^\vee\backslash G^\vee/H$ to the relevant orbits.

\begin{prop}\label{reverse order}
    The bijection $\Phi_X$ reverses the weak Bruhat order. 
\end{prop}

\begin{proof}

It is enough to show that for $s$ a simple reflection and for $x\in B\backslash X$ we have

$\Phi_X(sx)<\Phi_X(x)$ if and only if $sx>x$.

First, we notice that $\Phi_X$ switches the maximal and minimal elements in the weak Bruhat order. The result follows because to check whether $sx>x$ or $sx<x$ it is enough to know their distance from either the maximal or minimal element.

\end{proof}
\end{section}

\begin{section}{Borel Moore homology and the $W$ action}\label{s4}

In this section, we prove part 2 of Theorem \ref{relative KL}, using the results of the previous section.

We consider $H_{top}^{BM}(\Lambda)$ as a module over $H_{top}^{BM}(St)$. The module structure is given by the usual convolution construction as in Subsection \ref{convultion}. 

By Theorem \ref{classic} we know that $H_{top}^{BM}(St)\cong \C[W]$. The algebra $C[W]$ also acts on $\C[B\backslash X]$ and we will prove that these two modules are isomorphic after multiplication by $sgn$. 

In \cite{finkelberg2023lagrangiansubvarietieshypersphericalvarieties} a different conjecture is made about the module $H_{top}^{BM}(\Lambda)$. They conjecture that it is isomorphic to the module $H_{top}^{BM}(T^*X\times_{\mathfrak{g}^{*}}T^*(G/B))$. We also prove a version of this conjecture by showing that the module $H_{top}^{BM}(T^*X\times_{\mathfrak{g}^{*}}T^*(G/B))$ is isomorphic to $\C[B\backslash X]$. This isomorphism does  not play a role in the rest of the paper.

Let $\s_0\subset G^\vee/H\times \B$ be the unique minimal $G^\vee$ relevant orbit. The existence of $\s_0$ is guaranteed by Proposition \ref{reverse order} and the fact that $X$ has a unique open Borel orbit.

Recall that $\tau:\Lambda\rightarrow G^\vee/H\times \B$ is the projection on the first two coordinates. 

Let $\Lambda_0=\tau^{-1}(\s_0)\subset \Lambda$. The variety $\Lambda_0$ is one of the irreducible components of $\Lambda$. We denote by $\Lambda_0$ also the corresponding element of $H_{top}^{BM}(\Lambda)$.

We chose $\s_0$ to correspond (under $\Phi_X$) to the open Borel orbit of $X$, denote this open orbit by $\mathcal{U}_{max}$. The module $\C[B\backslash X]$ is generated over $\C[W]$ by $\mathcal{U}_{max}$. Another way to say this is to say that the map from $\C[W]$ to $\C[B\backslash X]$ given by action on  $\mathcal{U}_{max}$ is a surjective map of $C[W]$ modules.

\begin{prop}\label{map BM}
    Consider the map $\C[W]\rightarrow H_{top}^{BM}(\Lambda)$ given by acting on $\Lambda_0$. It factors through the map $\C[W]\rightarrow \C[B\backslash X]\otimes sgn$ given by acting on $\mathcal{U}_{max}$ with a sign twist. 
\end{prop}

\begin{proof}
    To prove the proposition it is enough to show that the stabilizer of $\mathcal{U}_{max}$ acts trivially on $\Lambda_0$ up to a sign. Let $s_1,...,s_{2n-1}\in W$ be the usual simple reflections. It is not hard to see that the stabilizer of  $\mathcal{U}_{max}$ is generated by $s_{2i-1},s_{2i}s_{2i-1}s_{2i+1}s_{2i}$ for $1\leq i\leq n$. One way to see this is by noting that under the correspondence between $B\backslash X$ and partitions of $1,...,2n$ into pairs, $\mathcal{U}_{max}$ corresponds to the pairing $\{1,2\},\{3,4\},...,\{2n-1,2n\}$.

    Thus, it is enough to show that $s_{2i-1}$ acts as $-1$ on $\Lambda_0$ and that the actions of $s_{2i-1}s_{2i}$ and $s_{2i+1}s_{2i}$ on $\Lambda_0$ agree.

    First let us recall how $s_i$ acts. We have in $\B\times \B$ the orbit $\s_s$ of Borel subgroups in relative position $s$. We have the projection $\pi$ from $St$ to $\B\times \B$ and we have $St_s=\overline{\pi^{-1}(\s_s)}\subset St$. Taking convolution with $St_s$ coincides with acting by $s+1\in \C[W]$.

    Consider $M^\vee\times_{\mathfrak{g}^{\vee*}}T^*(G/B)\times_{\mathfrak{g}^{\vee*}}T^*(G/B)$, there are two projections $\pi_1,\pi_2$ to $\Lambda$ and one projection $\pi_3$ to $St$.

    By definition $St_s*\Lambda_0=\pi_{2*}(\pi^*_1\Lambda_0\cap \pi_3^*St_s)$.

    Write $St_s*\Lambda_0=\sum a_iC_i$ as a sum of irreducible components of $\Lambda$. For every $C_i$ with $a_i\neq 0$, $\tau(C_i)\subset \overline{s\s_0}$.

 For $s=s_{2i-1}$ we have $s\s_0=\s_0$ so $\tau(C_i)= \s_0$. Thus, we have  $St_s*\Lambda_0=a\Lambda_0$ for some $a\in \C$.
 
 In fact, we claim that in this case $\pi^{-1}_1\Lambda_0\cap \pi_3^{-1}St_s=\emptyset$ so $a=0$ and $s\Lambda_0=-\Lambda_0$.

 Denote by $P^\vee_s$ the parabolic subgroups containing $B^\vee$ that corresponds to $s$, let $\mathfrak{p}^\vee_s$ be its Lie algebra. We have:

 $$\pi^{-1}_1\Lambda_0\cap \pi_3^{-1}St_s=\{(g_1H,g_2B^\vee,g_3B^\vee,\phi)|Hg_1^{-1}g_2B^\vee=\s_0,(g_2B^\vee,g_3B^\vee)\in \overline{\s_s},g_1^{-1}\phi|_{\mathfrak{h}}=\psi,g_2^{-1}\phi|_{\mathfrak{p}^\vee_s}=0 \}$$

We claim that for $g_1,g_2\in G^\vee$ with $Hg_1^{-1}g_2B^\vee=\s_0$ there is no $\phi\in \mathfrak{g}^{\vee*}$ such that $g_1^{-1}\phi|_{\mathfrak{h}}=\psi$ and $g_2^{-1}\phi|_{\mathfrak{p}^\vee_s}=0$. If there were such $\phi$ then $s\s_0$ would have been relevant but we know that $s\s_0<\s_0$ and $\s_0$ is a minimal relevant orbit.

  Now we need to show that the actions of $s_{2i-1}s_{2i}$ and $s_{2i+1}s_{2i}$ on $\Lambda_0$ agree.

 Let $\Lambda_{s_{2i}}=\overline{\tau^{-1}(s_{2i}\s_0)}$ be the irreducible component of $\Lambda$ that correspond to $s_{2i}$. We have that $s_{2i}\Lambda_0=a\Lambda_0+b\Lambda_{s_{2i}}$ for some $a,b\in \C$. It is enough to show that $s_{2i-1}$ and $s_{2i+1}$ act the same on $\Lambda_{s_{2i}}$. Both $s_{2i-1}\Lambda_{s_{2i}}$ and $s_{2i+1}\Lambda_{s_{2i}}$ are supported on $\Lambda_0,\Lambda_{s_{2i}}$ and on the irreducible componenet $\overline{\tau^{-1}(s_{2i+1}s_{2i}\s_0)}=\overline{\tau^{-1}(s_{2i-1}s_{2i}\s_0)}$. Both calculations are entirely symmetric so the actions of $s_{2i-1}$ and $s_{2i+1}$ on $\Lambda_{s_{2i}}$ are equal.  
\end{proof}

Now we have a map of $\C[W]$ modules $\Phi_f:\C[B\backslash X]\otimes sgn\rightarrow H_{top}^{BM}(\Lambda)$. We will show that it is an isomorphism. Both modules have the same dimension over $\C$ so it is enough to show that the map is surjective. This follows from the next proposition.

\begin{prop}\label{bm cyclic}
    The module $H_{top}^{BM}(\Lambda)$ is generated by $\Lambda_0$ over $H_{top}^{BM}(St)$. 
\end{prop}

\begin{proof}
    Recall that every irreducible component of $\Lambda$ is of the form $\Lambda_\s=\overline{\tau^{-1}(\s)}$ for some relevant orbit $\s\subset G^\vee/H\times \B$.
    We will show that the class of every irreducible component of $\Lambda$ is in $H_{top}^{BM}(St)\Lambda_0$ by induction on the Bruhat order.

    Let $\s$ be some relevant orbit. Assume that for any relevant orbit $\s'\subset\overline{\s}$ we have 
    
    $\Lambda_{\s'}\in H_{top}^{BM}(St)\Lambda_0$. Let $s\in W$ be a simple reflection such that $s\s$ is relevant and smaller than $\s$. 

    Consider $St_s*\Lambda_{s\s}$, it can be written as $St_s*\Lambda_{s\s}=\sum_{\s'}a_{\s'}\Lambda_{\s'} $ for $\s'$ relevant and contained in $\overline{\s}$.

    It is enough to show that $a_\s\neq 0$. 
    
    In fact, we have $a_\s=1$ as the map $\pi_2:\pi^{-1}_1\Lambda_{s\s}\cap \pi_3^{-1}St_s\rightarrow \Lambda$ is one to one over a generic point in $\Lambda_\s$.

\end{proof}

We deduce the following.
\begin{cor} The map given by Proposition \ref{map BM} $\Phi_f:\C[B\backslash X]\otimes sgn\rightarrow H_{top}^{BM}(\Lambda)$ is an isomorphism of $\C[W]$ modules.
\end{cor}

\begin{subsection}{A version of Conjecture 1.2 of \cite{finkelberg2023lagrangiansubvarietieshypersphericalvarieties}}

In this subsection we prove that $H_{top}^{BM}(T^*X\times_{\mathfrak{g}^{*}}T^*(G/B))\cong \C[B\backslash X]$ as $\C[W]$ modules. The proof is very similar to what we did before.

For this subsection we denote  $\Lambda^\vee=T^*X\times_{\mathfrak{g}^{*}}T^*(G/B)$. In this subsection we also use $St$ for the Steinberg variety of $G$ and not of $G^\vee$.

We have a map which we also denote by $\tau$, $\tau:\Lambda^\vee\rightarrow X\times G/B$.

The irreducible components of $\Lambda^\vee$ are given by $\overline{\tau^{-1}(\mathcal{U})}$ for $\mathcal{U}$ some $G$ orbit in $X\times G/B$. 

We have the minimal orbit in $X\times G/B$ which we denote by $\mathcal{U}_0$ and we have $\Lambda^\vee_0=\tau^{-1}(\mathcal{U}_0)$.

\begin{prop}
    Consider the map $\C[W]\rightarrow H_{top}^{BM}(\Lambda^\vee)$ given by acting on $\Lambda^\vee_0$. It factors through the map $\C[W]\rightarrow \C[B\backslash X]$ given by acting on $\mathcal{U}_0$. 
\end{prop}

\begin{proof}
    The proof is very similar to the proof of Proposition \ref{map BM}. The stabilizer of $\mathcal{U}_0$ is generated by $s_is_{i+n}$ for $1\leq i\leq n-1$ and by $s_{n}$.

    It is enough to check that $s_n$ acts trivially on $\Lambda^\vee_0$ and that $s_i,s_{i+n}$ act in the same way. The fact that $s_i$ and $s_{i+n}$ act the same on $\Lambda^\vee_0$ follows from symmetry as in the proof of Proposition \ref{map BM}. 

    We check that $s=s_n$ acts trivially. Like in the proof of Proposition \ref{map BM}, the fact that $s\mathcal{U}
    _0=\mathcal{U}
    _0$ implies that there is $a\in \C$ such that $St_s*\Lambda^\vee_0=a\Lambda^\vee_0$. We claim that $a=2$.

 Denote by $P_s$ the parabolic subgroups containing $B$ that corresponds to $s$. Let $\mathfrak{p}_s$ be the Lie algebra of $P_s$. We have:
     $$\pi^{-1}_1\Lambda^\vee_0\cap \pi_3^{-1}St_s=\{(g_1Sp_{2n},g_2B,g_3B,\phi)|Sp_{2n}g_1^{-1}g_2B=\mathcal{U}_0,(g_2B,g_3B)\in \overline{\s_s},g_1^{-1}\phi|_{\mathfrak{sp}_{2n}=0},g_3^{-1}\phi|_{\mathfrak{p}_s}=0 \}$$

     The map $\pi_2$ to $\Lambda^\vee$ is forgetting the second coordinate, meaning forgetting $g_2B$. 

     Notice that for every $(g_1Sp_{2n},g_2B,g_3B,\phi)\in \pi^{-1}_1\Lambda^\vee_0\cap \pi_3^{-1}St_s$ we have $Sp_{2n}g_1^{-1}g_3B=\mathcal{U}_0$.
     
     Fix $g_1Sp_{2n}\in X$ and $g_3B$ such that $Sp_{2n}g_1^{-1}g_3B=\mathcal{U}_0$. For any $g_2B\in G/B$ with $g_2B$ and $g_3B$ in relative position $s$ we have $Sp_{2n}g_1^{-1}g_2B=\mathcal{U}_0$. Thus the fibers of $\pi_2:\pi^{-1}_1\Lambda^\vee_0\cap \pi_3^{-1}St_s\rightarrow \Lambda^\vee_0$ are all isomorphic to $\p^1$. The Euler characteristic of $\p^1$ is 2 and thus by Theorem 2.7.26 of \cite{Chriss1997RepresentationTA} we have $a=2$. 
    
\end{proof}

We also need to show that $\Lambda^\vee_0$ generates $H_{top}^{BM}(\Lambda^\vee)$. Again the proof is the same as in the pervious case, see Proposition \ref{bm cyclic}. 

We deduce a special case of a version of conjecture 1.1.2 of \cite{finkelberg2023lagrangiansubvarietieshypersphericalvarieties}.

\begin{prop}
    $H_{top}^{BM}(\Lambda^\vee)\cong H_{top}^{BM}(\Lambda)\otimes sgn$ as $\C[W]$ modules.
\end{prop}

\end{subsection}
    
\end{section}

\begin{section}{Iwahori orbits and the module $S(X)^I$}\label{s5}

In this Section, we describe the $I$ orbits on $X$. We use them to give an explicit discreption of the module $S(X)^I$.

Let $F$ be a non-archimedean local field, let $O$ be its ring of integers, and let $u\in O$ be a uniformaizer. Denote by $\nu$ the valuation of $F$.

We have $I\subset GL_{2n}(O)$ the Iwahori subgroup of integer matrices whose entries below the diagonal are divisible by $u$.

\begin{prop}\label{I orbits}
    We already saw that $X$ is isomorphic to the space of anti-symmetric invertible matrices. Under this isomorphism, the $I$ orbits on $X$ are represented by monomial matrices whose non zero entries above the main diagonal are powers of $u$.
\end{prop}

\begin{proof}
    This is a simple computation, we present it for completion.

    Let $\<,\>:F^{2n}\times F^{2n}$ be the anti-symmetric bilinear form preserved by $Sp_{2n}(F)$.

    Let $x\in X$, we consider $x$ as Gram matrix of some basis $v_1,...,v_{2n}\in F^{2n}$. 
    
    We have $x_{ij}=\<v_i,v_j\>$. Any element $g\in G$ acts by $v_i\mapsto gv_i$ which changes $x$ accordingly. 

    Let $1\leq i,j\leq 2n$ be such that $\nu(\<v_i,v_j\>)$ is minimal and such that there are no other pair $i',j'$ with $\nu(\<v_i,v_j\>)=\nu(\<v_{i'},v_{j'}\>)$ and $i'\geq i,j'\geq j$. We claim that acting by $I$ we can change the vectors such that $v_i,v_j$ become orthogonal to all other vectors. Let $1\leq k\leq 2n$ with $k\neq i,j$.  Consider $v_k'=v_k-v_i\frac{\<v_j,v_k\>}{\<v_j,v_i\>}+v_j\frac{\<v_i,v_k\>}{\<v_j,v_i\>}$. Notice that $\<v_k',v_i\>=\<v_k',v_j\>=0$. The coefficients of $v_i,v_j$ are integers. If $k>i$ then $\nu(\frac{\<v_j,v_k\>}{\<v_j,v_i\>})>0$, similarly for $j$. Set $v_i'=v_i$ and $v_j'=v_j$. The coordinate change $v_1,...,v_{2n}\rightarrow v_1',...,v_{2n}'$ can be obtained by an action of $I$. 

    We can also change $v_i,v_j$ such that if $i<j$ then $\<v_i,v_j\>$ is a power of $u$.

    After the coordinate change we can ignore $v_i,v_j$ as they are orthogonal to the other vectors and continue in the same way to get our desired result.

\end{proof}

Using this result we embed $B\backslash X$ into $I\backslash X$ as anti-symmetric monomial matrices whose entries are either $1$ or $-1$.

Let $T\subset G$ be the torus of diagonal elements, let $T^0\subset T$ be the subgroup of integer elements. 

Let $W_{aff}=N_G(T)/T^0$ be the extended affine Weyl group of $G$. The group $W_{aff}$ is isomorphic to the group of monomial matrices whose non zeros entries are powers of $u$.  

In \cite{my} an action of $W_{aff}$ on the $I$ orbits on $X$ is constructed. With our model of $W_{aff}$ and the $I$ orbits on $X$ it is easy to describe the action explicitly. Let $w\in W_{aff}$ and $x\in X$ represented by a monomial matrix as in Proposition \ref{I orbits}, the action of $w$ on $x$ is given by $w\times x= wxw^t$ which is again a monomial matrix like in Proposition \ref{I orbits}. 

\begin{defn}
    Let $\C[I\backslash X]$ be the free vector space generated by $I\backslash X$. The action of $W_{aff}$ on $I\backslash X$ induces an action of $\C[W_{aff}]$ on $\C[I\backslash X]$.
\end{defn}

Let $H(G,I)$ be the affine Hecke algebra of $G$, here we consider the generic version where the variable $q$ is just a formal variable and not the size of a residue field. 

We also have the generic version of the module $S(X)^I$ over $H(G,I)$ constructed in \cite{my}. We will recall the construction for the specific case of $X=GL_{2n}/Sp_{2n}$.

\begin{prop}
 $S(X)^I$ is generated over $\C[q,q^{-1}]$ by characteristic functions of $I$ orbits, $I\backslash X$. Let $l$ be the usual length function on $W_{aff}$. We have a length function $l_\sigma$ on $I\backslash X$ (Constructd in \cite{my}). By Proposition \ref{I orbits} every $I$ orbit is represented by an anti-symmetric element of $W_{aff}$, we have $l_\sigma(w)=l(w)-1$ (see Proposition 9.3 in \cite{my}).

Let $s\in W_{aff}$ be a simple reflection, we have a corresponding element $T_s\in H(G,I)$. Let $x\in I\backslash X$ be an $I$ orbit. We have either $sx=x$, $l_\sigma(sx)<l_\sigma(x)$, or $l_\sigma(sx)>l_\sigma(x)$. If $sx=x$ we have $T_s1_x=q1_x$, if $l_\sigma(sx)<l_\sigma(x)$ we have  $T_s1_x=(q-1)1_x+q1_{sx}$, and if $l_\sigma(sx)>l_\sigma(x)$ we have $T_s1_x=1_{sx}$. 

Again $q$ is just a formal variable. This gives a well defined module over $H(G,I)$ as was proven in \cite{my}.
   
\end{prop}

The group $W_{aff}$ acts transitively on $I\backslash X$. This implies that the module $S(X)^I$ is generated over $H(G,I)$ by the characteristic function of $\mathcal{U}_{max}\in B\backslash X$, considered as an element of $I\backslash X$ via the defined embedding. We denote this function by $f_{max}$.

We recall the Bernstein description of $H(G,I)$.

\begin{defn}
    For any $w\in W_{aff}$ we denote by $T_w\in H(G,I)$ the element corresponding to the characteristic function of $w$. 
    Let $X_*(T)$ be the lattice of coroots of $T$.
    The choice of a Borel $B$ that contains $T$ gives a choice of dominant coroots. Specifically, for our choice of $B$ being the Borel of upper triangular matrices in $G$, a coroot $\lambda(t)=diag(t^{\lambda_1},..., t^{\lambda_{2n}})$ is dominant if and only if $\lambda_1\geq ...\geq \lambda_{2n}$. 

    Every $\lambda\in X_*(T)$ we be written as a ratio of two dominant elements $\lambda^+,\lambda^-$, $\lambda=\lambda^+\cdot (\lambda^-)^{-1}$. We then define $\theta_\lambda=q^\frac{l(\lambda^-)-l(\lambda^+)}{2}(T_{\lambda^+(u)})(T_{\lambda^-(u)})^{-1}$.
\end{defn}

The elements of the form $\theta_\lambda$ generate a commutative sub algebra of $H(G,I)$ called the Bernstein sub algebra, we denote it by $\Theta$.

The following result due to Bernstein was first mentioned in \cite{Bernsteins_presentation}.

\begin{theorem}\label{Ber_presentation}
    Let $H_f$ be the finite Hecke algebra, it is the sub algebra of $H(G,I)$ generated by $T_w$ for $w\in W$ in the finite Weyl group. 

    $H(G,I)$ is isomorphic to the algebra generated by $H_f$ and $\Theta$ with the following relations.

    For $s=s_\alpha\in W$ a simple reflection and $\lambda\in X_*(T)$ we have $$T_s\theta_{s(\lambda)}-\theta_\lambda T_s=(1-q)\cdot \frac{\theta_\lambda-\theta_{s(\lambda)}}{1-\theta_{-\alpha^\vee}}=(1-q)\sum^{\alpha(\lambda)-1}_{i=0} \theta_\lambda\theta^i_{-\alpha^\vee}$$
\end{theorem}

Now back to $S(X)^I$, we describe the annihilator of $f_{max}$. First we need a notation.
\begin{Not}
    Let $L_{max}$ be the lattice of coroots of the form $\lambda(t)=(t^{a_1},t^{-a_1},t^{a_2},t^{-a_2},...,t^{a_n},t^{-a_n})$ for some $(a_1,...,a_n)\in\Z^n$.
\end{Not}

\begin{prop}\label{ideal generaotrs}
    Consider the action of $H(G,I)$ on $f_{max}$. The annihilator of $f_{max}$ is generated by the elements of the form:
    \begin{enumerate}
        \item $T_{s_{2i-1}}-q$.
        \item $T_{s_{2i-1}}T_{s_{2i}}-T_{s_{2i+1}}T_{s_{2i}}$.
        \item $\theta_{\lambda}-q^\frac{l(\lambda^-)-l(\lambda^+)}{2}$ for $\lambda\in L_{max}$.
    \end{enumerate}
\end{prop}

    To prove Proposition \ref{ideal generaotrs} we use the boundary degeneration of $X$. In Appendix \ref{A1} we deal with the general theory of boundary degenerations of a symmetric space and its relation to the action of the Hecke algebra.
    
    Let $X_\emptyset=G/S_\emptyset$ be the most degenerate boundary degeneration of $X$. Let $P_X$ be the standard parabolic corresponding to the the set of simple roots $\{s_1,s_3,...,s_{2n-1}\}$. Let $U^-_X$ be the unipotent radical of the opposite parabolic of $P_X$.  Explicitly, it is the group of matrices of the form $U_X^-=\{\begin{pmatrix}
        1 & 0 & 0& 0&...\\
        0 & 1& 0 & 0&...\\
        * & * & 1& 0&...\\
        * & * & 0& 1&...\\
        ...& ...& ...& ...&...
    \end{pmatrix}\}$ .
    
    We can embed $SL_2^n$ into the Levi of $P_X$ as diagonal $2\times 2$ block matrices with determinant 1. Denote this group by $L_X$. We have $S_\emptyset=U^-_XL_X$. 
    
    There is a Bernstein map $e:S(X_\emptyset)\rightarrow S(X)$ (see Theorem \ref{theorem3}). It induces a map of $H(G,I)$ modules which we also denote by $e:S(X_\emptyset)^I\rightarrow S(X)^I$. By Proposition \ref{prop:matching} we can find $w\in W$ such that $e(1_{IwS_\emptyset})=f_{max}$, by conjugating $Sp_{2n}$ we can assume that $w=1$. We take $Sp_{2n}$ to be the group that fixes the symplectic form $$\omega(v,w)=v_1w_2-v_2w_1+v_3w_4-v_4w_3+...+v_{2n-1}w_{2n}-v_{2n}w_{2n-1}$$ 
    
    Denote $f_\emptyset=1_{IwS_\emptyset}=1_{IS_\emptyset}$.

    We compute the annihilator of $f_\emptyset$ and the kernel of $e$ and together we get the annihilator of $f_{max}$.

    \begin{prop}
    The annihilator of $f_\emptyset\in S(X_\emptyset)^I$ is generated by the elements of the form:
    \begin{enumerate}
        \item $T_{s_{2i-1}}-q$.
        \item $\theta_{\lambda}-q^\frac{l(\lambda^-)-l(\lambda^+)}{2}$ for $\lambda\in L_{max}$.
    \end{enumerate}
    \end{prop}

    \begin{proof}
        We begin with checking that these elements indeed annihilate $f_\emptyset$.

        Let $s=s_{2i-1}$, by Proposition \ref{action} to see that $T_sf_\emptyset=qf_\emptyset$ we need to find a sequence $t_m\in T$, $t_m\rightarrow\infty$ such that $Ist_mSp_{2n}=It_mSp_{2n}$. This holds for any $t\in T$.

        Now, let $\lambda\in L_{max}$, we need to show that $\theta_{\lambda}f_\emptyset= q^\frac{l(\lambda^-)-l(\lambda^+)}{2}f_\emptyset$. We claim that for every $\lambda\in X_*(T)$, $\theta_{\lambda}f_\emptyset=\theta_{\lambda}1_{IS_\emptyset}= q^\frac{l(\lambda^-)-l(\lambda^+)}{2}1_{I\lambda(u)S_\emptyset}$. 
        
        In order to prove this, it is enough to show that if $\lambda\in X_*(T)$ is dominant and $t'\in T$ is arbitrary then $1_{I\lambda(u)I}1_{It'S_\emptyset}=1_{I\lambda(u)t'S_\emptyset}$. We need to find $t_m\rightarrow\infty$ such that $1_{I\lambda(u)I}1_{It't_mSp_{2n}}=1_{I\lambda(u)t't_mSp_{2n}}$. We can take $t_m\rightarrow\infty$ such that $t't_m$ is dominant and then this holds.

        Denote the sub algebra of $\Theta$ generated by $\theta_{\lambda}-q^\frac{l(\lambda^-)-l(\lambda^+)}{2}$ for $\lambda\in L_{max}$ by $\Theta_X$.

        We need to show that the mentioned elements generate the annihilator of $f_\emptyset$.

        Assume that $h\in H(G,I)$ satisfies $hf_\emptyset=0$. We can write $h=\sum_{w,\lambda} a_\lambda^wT_w\theta_\lambda$ for $w\in W$, $\lambda\in X_*(T)$ and coefficients $a_\lambda^w\in \C[q^{\pm 1}]$. After changing the coefficients by a factor of a power of $q$ we have $hf_\emptyset= \sum a_\lambda^w T_w1_{I\lambda(u)S_\emptyset}$.

        For fixed $\lambda$ and for every $w\in W$, $T_w1_{I\lambda(u)S_\emptyset}$ is supported on $IW\lambda(u)S_\emptyset$. We claim that for $\lambda,\lambda'$ such that $I\lambda(u)S_\emptyset\neq I\lambda(u)'S_\emptyset$ we have $IW\lambda(u)S_\emptyset\cap IW\lambda(u)'S_\emptyset=\emptyset$. To prove this it is enough to show that for every $\lambda''$ dominant enough if $I\lambda(u)\lambda(u)''Sp_{2n}\neq I\lambda(u)'\lambda(u)''Sp_{2n}$ then $IW\lambda(u)\lambda(u)''Sp_{2n}\cap IW\lambda(u)'\lambda(u)''Sp_{2n}=\emptyset$. 
        
        In any $I$ orbit in $IW\lambda(u)\lambda(u)''Sp_{2n}$ there is a unique anti-symmetric matrices whose non zero entries are equal to (up to a sign) $\lambda_1(u)\lambda''_1(u)\lambda_2(u)\lambda''_2(u),\lambda_3(u)\lambda''_3(u)\lambda_4(u)\lambda''_4(u),...$. 
        
        For $IW\lambda(u)\lambda(u)''Sp_{2n}$ the same is true with $\lambda_i$ replaced by $\lambda'_i$. By our assumption not all of these values are the same, and thus the $I$ orbits are different.

        Thus, in the expression $\sum_{w,\lambda} a_\lambda^w T_w1_{I\lambda(u)S_\emptyset}=0$ we can group the elements by the functions $1_{I\lambda(u)S_\emptyset}$ and for every fixed $\lambda$ we have $\sum_w a_\lambda^w T_w1_{I\lambda(u)S_\emptyset}=0$. 

        We claim that if $h'\in H_f$ and $h'1_{I\lambda(u)S_\emptyset}=0$ then $h'$ is in the ideal generated by $T_{s_{2i-1}}-q$. It is easy to see that $T_{s_{2i-1}}-q$ acts as zero on $1_{I\lambda(u)S_\emptyset}$. The dimension of $H_f1_{I\lambda(u)S_\emptyset}$ is equal to the number of $I$ orbits in $IW\lambda(u)S_\emptyset$. The number of such orbits is equal to the size of the coset space $W/\langle s_{2i-1}|i=1,...,2n\rangle$ and the result follows. 

        We proved that an element in the annihilator of $f_\emptyset$ modulo $\Theta_X$ can be written as sum of elements in $H_f(T_{s_{2i-1}}-q)\Theta$. To complete the proof we use Theorem \ref{Ber_presentation}.

        For $s=s_{2i-1}$ and $\alpha=\alpha_{2i-1}$ we have $$(T_s-q)\theta_\lambda=\theta_{s(\lambda)}T_s+(1-q)\frac{\theta_{s(\lambda)}-\theta_{\lambda}}{1-\theta_{-\alpha^\vee}}-q\theta_\lambda=\theta_{s(\lambda)}(T_s-q)+q\theta_{s(\lambda)}+(1-q)\frac{\theta_{s(\lambda)}-\theta_{\lambda}}{1-\theta_{-\alpha^\vee}}-q\theta_\lambda$$
        
        We have $q\theta_{s(\lambda)}+(1-q)\frac{\theta_{s(\lambda)}-\theta_{\lambda}}{1-\theta_{-\alpha^\vee}}-q\theta_\lambda=\frac{\theta_{s(\lambda)}-\theta_{\lambda}}{1-\theta_{-\alpha^\vee}}(q-q\theta_{-\alpha^\vee}+1-q)=\frac{\theta_{s(\lambda)}-\theta_{\lambda}}{1-\theta_{-\alpha^\vee}}(-q\theta_{-\alpha^\vee}+1)$. 

        Notice that $(-q\theta_{-\alpha^\vee}+1)\in \Theta_X$.
        
        We see that $H_f(T_{s_{2i-1}}-q)\Theta$ is in $H(G,I)\Theta_x+H(G,I)(T_{s_{2i-1}}-q)$ which completes the proof.
        
    \end{proof}

    Next, we describe the kernel of $e:S(X_\emptyset)^I\rightarrow S(X)^I$.

    \begin{prop}
        Let $e:S(X_\emptyset)^I\rightarrow S(X)^I$ be the Bersntein morphism. Its kernel is generated by the elements of the form $(T_{s_{2i-1}}T_{s_{2i}}-T_{s_{2i+1}}T_{s_{2i}})f_\emptyset$. 
    \end{prop}

    \begin{proof}
        It is clear that these elements are in the kernel as $(T_{s_{2i-1}}T_{s_{2i}}-T_{s_{2i+1}}T_{s_{2i}})f_{max}=0$.

        Consider $H_ff_\emptyset$ modulo the space of functions in $H_f(T_{s_{2i-1}}T_{s_{2i}}-T_{s_{2i+1}}T_{s_{2i}})f_\emptyset$. We can find a basis of this quotient represented by characteristic functions $1_{IwS_\emptyset}$ for $w\in W$ such that $e(1_{IwS_\emptyset})=1_{IwSp_{2n}}$. Choose for every element in this basis some $w\in W$ such that $1_{IwS_\emptyset}$ represents the basis element. Denote by $W^X\subset W$ the collection of these representatives. Denote $V=span\{1_{IwS_\emptyset}|w\in W^X\}$. We claim that $e$ is injective on $\Theta V$. This implies the desired result.

        Reducing mod $q-1$ we get a map of $I$ orbits $I\backslash G/S_\emptyset\rightarrow I\backslash X$. The image of $\Theta V$ under this reduction is the space of orbits of the form $ItwS_\emptyset$ for $w\in W^X$. It is easy to check that on this set the map $I\backslash G/S_\emptyset\rightarrow I\backslash X$ is injective.

        Now assume that $f\in \Theta V$ satisfies $e(f)=0$. Write $f=\sum_{w\in W^X,\lambda\in X_*(T)} a_\lambda^w(q)\theta_\lambda1_{IwS_\emptyset}$ for some $a_\lambda^w\in \C[q^{\pm 1}]$. We may assume that $a_\lambda^w$ are polynomial in $q$. We know that reducing mod $q-1$ the map $e$ is injective, this means that $q-1|a_\lambda^w(q)$. We can divide by $q-1$ and get an element in the kernel of $e$ whose coefficients have lower degree. We can continue like this and get that $f=0$. 
    \end{proof}

Now we can prove Proposition \ref{ideal generaotrs}.

\begin{proof}[Proof of Proposition \ref{ideal generaotrs}]

Let $h\in H(G,I)$ be such that $hf_{max}=0$, we have $e(hf_\emptyset)=0$. Thus, we can find $h'\in H(G,I)$ in the ideal generated by the elements of the form $(T_{s_{2i-1}}T_{s_{2i}}-T_{s_{2i+1}}T_{s_{2i}})$ such that $hf_\emptyset=h'f_\emptyset$. This means that $h-h'$ is in the annihilator of $f_\emptyset$, so it is in the ideal generated by $\Theta_X$ and the elements of the form $T_{s_{2i-1}}-q$.  
    
\end{proof}

Let us recall here the definition of the Iwahori Matsumoto involution on $H(G,I)$. It will be used in Section \ref{s7}.

\begin{defn}\label{involution}
    There is a unique involution $IM:H(G,I)\rightarrow H(G,I)$ which satisfies
    \begin{enumerate}
        \item $IM(T_s)=-T^{-1}_s=-\frac{T_s+1}{q}+1$ for $s\in \Delta$ a simple reflection in the finite Weyl group.
        \item $IM(\theta_\lambda)=\theta_{\lambda^{-1}}$ for $\lambda\in X_*(T)$.
    \end{enumerate}

    For any $H(G,I)$ module $V$ we denote by $IM(V)$ the $H(G,I)$ module with the same vector space as $V$ and with the action given by $h\cdot_{IM(V)} v:=IM(h)\cdot_Vv$.
\end{defn}

The Iwahori Matsumoto involution specialized to $q=1$ is given by tensoring with the extended sign character $sgn_f$ defined below.

\begin{defn}
    Let $w\in W_{aff}$, we can write it as $w=w_0t$ for $w_0\in W$ and $t\in T/T^0$.
    Define $sgn_f:W_{aff}\rightarrow\{\pm 1\}$ by $sgn_f(w)=sgn_f(w_0t)=sgn(w_0)$.
\end{defn}

\end{section}

\begin{section}{Cellular fibration argument}\label{s6}

In this Section, we describe a filtration on $\Lambda$ and relate it to a filtration on a smaller variety $\Lambda_H$ which satisfies the cellular fibration lemma (see Lemma \ref{cfl}).

For any integer $m$ we denote by $B_{m}$ the Borel subgroup of $GL_{m}$ consisting of upper triangular matrices. We also denote by $\B_m=GL_{m}/B_{m}$ the flag variety of $GL_m$. 

Let $\Lambda_H=\{(gB^\vee,\phi)|gB^\vee\in \B,\phi\in \mathfrak{g}^{\vee*},\phi|_\mathfrak{h}=\psi,(g^{-1}\phi)|_\mathfrak{b^\vee}=0\}$.

\begin{prop}
    We have $\Lambda\cong G^\vee\times^H \Lambda_H$ as $G^\vee$ varieties.
\end{prop}

\begin{proof}
    Consider the map from $G^\vee\times^H \Lambda_H$ to $\Lambda$ that sends $(g_1,g_2B^\vee,\phi)$ to $(g_1H,g_1g_2B^\vee,g_1\phi)$. Clearly, it is well defined and is an isomorphism.
\end{proof}

We will show that $\Lambda_H$ satisfies the cellular fibration lemma over $GL_n/B_n$.

We have the projection $\tau:\Lambda_H\rightarrow \B_{2n}$. Let $W$ be the $n$ dimensional vector space inside $\C^{2n}$ of vectors $v$ such that $v_i=0$ for $i>n$. Let $U_H$ be the unipotent radical of $H$. The group $H$ acts on $W$ and its action factors through $H/U_H\cong GL_n$.

We define a map $f:\B_{2n}\rightarrow\B_n$. 

We think about an element of $\B_{2n}$ as a full flag $\mathcal{F}=(0=V_0\subset V_1\subset...\subset V_{2n}=\C^{2n})$. We define a flag by erasing repetition in $0=V_0\cap W\subset V_1\cap W\subset...\subset V_{2n}\cap W=W$. This is the element $f(\mathcal{F})$ of $\B_n$.

\begin{prop}\label{affine fibers}
    Let $\s$ be a $H$ orbit on $\B_{2n}$. The map $f:\s\rightarrow\B_n$ is an affine fibration. Meaning, it is a fibration and the fibers are isomorphic to some $\mathbb{A}^k$.
\end{prop}

\begin{proof}
    The group $H$ acts acts transitively on $\B_n$ so it is enough to compute the fiber over a specific point.

    Let $e_1,...,e_{2n}$ be the standard basis of $\C^{2n}$.
    
    Consider the flag $\mathcal{G}=(0\subset \C e_1\subset \C e_1\oplus \C e_2\subset ...\subset \C e_1\oplus... \oplus \C e_n=W)\in \B_n$.

    Let $B\in \B_{2n}$ the Borel subgroup which corresponds to the flag $$(0\subset \C e_1\subset \C e_1\oplus \C e_2\subset ...\subset \C e_1\oplus... \oplus \C e_{2n})$$
        
    Let $\mathcal{F}=(0=V_0\subset V_1\subset...\subset V_{2n}=\C^{2n})\in \B_{2n}$ be a flag such $f(\mathcal{F})=\mathcal{G}$.

    It is enough to show that the orbit of $\mathcal{F}$ under the group of matrices of the form
    
    $L=\{\begin{pmatrix}
        b & a \\
        0 & b
    \end{pmatrix}|b\in B_n,a\in M_{n\times n}\}$ is an affine space. Denote this space by $L\mathcal{F}$.

    There is another map $g:\B_{2n}\rightarrow \B_n$ given by taking the flag mod $W$ to get a full flag in $V/W$. The fibers of $L\mathcal{F}$ under this map are all isomorphic to $W^n$. The action of $H$ on $g(L\mathcal{F})$ factors through $B_n$. It is well known that the orbits of $B_n$ on $\B_n$ are affine. Overall, we get that $L\mathcal{F}$ is affine, it is the product of $W^n$ and some orbit of $B_n$ on $\B_n$. 
\end{proof}

\begin{Remark}
    Proposition \ref{affine fibers} is an analogue of Lemma 6.2.5 in \cite{Chriss1997RepresentationTA}. 
\end{Remark}

\begin{defn}
    Chose some linear order on the relevant $H$ orbits on $\B_{2n}$ that is compatible with the Bruhat order, $\s_0<\s_1<...<\s_m=\s_{max}$. Here, $\s_0$ is the minimal relevant orbit and $\s_{max}$ is the open orbit. 

     Define $\Lambda^i_H=\cup_{j\leq i} \tau^{-1}(\s_j)$ and $\Lambda_i=G^\vee\times^H \Lambda^i_H$.

\end{defn}

\begin{prop}
    The map $f\circ \tau:\Lambda_H\rightarrow \B_n$ is a $H$ equivariant map that satisfies the cellular fibration lemma.
\end{prop}

\begin{proof}

    We claim that $\Lambda_H^0\subset ...\subset \Lambda^m_H=\Lambda_H$ is a cellular fibration. We need to check the three conditions appearing in Lemma \ref{cfl}.

    It is clear that each $\Lambda^i_H$ is $H$ stable and closed. We need to check that $f\circ \tau:\Lambda^i_H\rightarrow \B_n$ is a locally trivial fibration. $H$ acts transitively on $\B_n$ so it is enough to show the second conditions, meaning that each $\tau^{-1}(\s_j)$ is an affine fibration over $\B_n$. All these affine fibrations have to glue to a locally trivial fibration over $\B_n$.

    Now we turn to looking at $f\circ\tau :\tau^{-1}(\s_j)\rightarrow \B_n$. The map $\tau:\tau^{-1}(\s_j)\rightarrow \s_i$ is an affine fibartion. Fix $g\in G^\vee$ such that $\s=HgB^\vee$. The fibers of $\tau:\tau^{-1}(\s_j)\rightarrow \s_i$ are isomorphic to $\{\phi\in \mathfrak{g}^{\vee *}|(g^{-1}\phi)|_\mathfrak{h}=\psi,\phi|_\mathfrak{h}=\psi\}$ which is an affine space.

    By Proposition \ref{affine fibers}, the map $f:\s_j\rightarrow\B_n$ is an affine fibration. The composition of two affine fibrations is an affine fibration.
    
    Lastly, we need to show that $K^H(\B_n)$ is a free $R(H)$ module. 

    Let $U_H$ be the unipotent radical of $H$. By Subsection 5.2.18 of \cite{Chriss1997RepresentationTA} we may replace $H$ by $H/H_U\cong GL_n$. We know that $K^{GL_n}(\B_n)=K^{GL_n}(GL_n/B_n)=R(B_n)$ which is a free $R(GL_n)$ module.
    
\end{proof}

From the cellular fibration Lemma we deduce the following.

\begin{cor}
    $K^H(\Lambda_H)$ is a free $R(H)$ module. In fact, it is a freely   generated by the structure sheaves of its irreducible components as a $K^H(\B_n)=R(B_n)$ module.

    We also have that $K^{G^\vee}(\Lambda)=K^H(\Lambda_H)$ is a $R(B_n)$ module freely generated by the structure sheaves of its irreducible components.
\end{cor}

We have a version of the above corollary also for $G^\vee\times \C^\times$ equivariant $K$ theory.

First we need to describe the $\mathbb{G}_m$ action on $M^\vee$.

The action of $\mathbb{G}_m$ on $\mathfrak{g}^{\vee*}$ is given by $z,v\rightarrow z^2v$.  

\begin{Remark}\label{linear action}
    This differs from the action defined in \cite{Chriss1997RepresentationTA} and \cite{Kazhdan1987ProofOT}. The action defined in these works is $z,v\rightarrow zv$. 
\end{Remark}

\begin{defn}
    Consider the map $\rho:\mathbb{G}_m\rightarrow G^\vee$ given by $z\rightarrow diag(z,...,z,z^{-1},...,z^{-1})$ where $diag(z,...,z,z^{-1},...,z^{-1})$ is a diagonal matrix whose first $n$ elements are $z$ and last $n$ elements are $z^{-1}$.
    
    Notice that $\rho(z)$ normalizes $H$. For $h=\begin{pmatrix}
        g & a  \\
        0 & g
        \end{pmatrix}\in H$,  $\rho(z)h\rho(z)^{-1}=\begin{pmatrix}
        g & z^2a  \\
        0 & g
        \end{pmatrix}$. 

    We define an action of $\mathbb{G}_m$ on $M^\vee$, $z\in \mathbb{G}_m$ sends $(gH,\phi)\in M^\vee$ to $(g\rho(z)H,z^2\phi)\in M^\vee$.

    This action obviously commutes with the action of $G^\vee$.
\end{defn}

We gave a description of $K^{G^\vee}(\Lambda)=K^H(\Lambda_H)$ using the cellular fibration lemma. Now we have a $G^\vee\times\mathbb{G}_m$ action on $\Lambda$ and we want to have a similar description for $K^{G^\vee\times\mathbb{G}_m}(\Lambda)$. For that we need to enlarge $H$. 

The proof of the following proposition is straightforward.

\begin{prop}\label{semi direct product}
    We have a $\mathbb{G}_m$ action on $H$, $z\in \mathbb{G}_m$ acts as conjugation by $\rho(z)$. Define $\Tilde{H}=H\rtimes \mathbb{G}_m$ with the described action. 

    There is a $\Tilde{H}$ action on $\Lambda_H$ extending the $H$ action such that $$(1,z)\cdot(gB,\phi):=(\rho(z^{-1})g\rho(z)B,z^2\phi)$$

    We can embed $\Tilde{H}$ into $G^\vee\times\mathbb{G}_m$ and we have  $\Lambda=\Lambda_H\times^{\Tilde{H}}(G^\vee\times\mathbb{G}_m)$
\end{prop}

Using the same cellular fibration argument and the identity $K^{G^\vee\times\mathbb{G}_m}(\Lambda)=K^{\Tilde{H}}(\Lambda_H)$ we get that $K^{G^\vee\times\mathbb{G}_m}(\Lambda)$ is a free $R(B_n\times \mathbb{G}_m)$ module generated by the structure sheaves of the irreducible components of $\Lambda$.

\end{section}

\begin{section}{The module $K^{G^\vee\times \C^\times}(\Lambda)$}\label{s7}

In this section, we study $K^{G^\vee\times \C^\times}(\Lambda)$ and prove part four of Theorem \ref{relative KL}. The same arguments also give part three of Theorem \ref{relative KL}.

Denote by $T^\vee$ the complex torus dual to $T$. We think about $T^\vee$ as the torus of complex diagonal matrices inside $G^\vee$.

First, we consider the action of $\Theta\subset H(G,I)$ on $K^{G^\vee\times \C^\times}(\Lambda)$.

\begin{defn}
    Let $\lambda$ be a character of $T^\vee$. We denote by $L_\lambda$ the corresponding line bundle on $\B$.
\end{defn}

Let $St_0\subset St$ be the preimage of the diagonal $\Delta \B$ under the projection $\tau:St\rightarrow \B\times \B$. Restricting $\tau$ to $St_0$ we have a map to $\tau:St_0\rightarrow\B$. Let $L_\lambda^{St}=\tau^*L_\lambda$. We can extend this to the entire $St$ and get an element of $K^{G^\vee\times\C^\times}(St)$ which we also denote by $L_\lambda^{St}$. We can think about the character $\lambda$ as a cocharacter of $T$. The element $\theta_\lambda^{-1}$ acts on $K^{G^\vee\times\C^\times}(\Lambda)$ by convolution with $L^{St}_\lambda$ (see Subsection 7.6 of \cite{Chriss1997RepresentationTA}).

We perform a similar construction for $\Lambda$.

\begin{defn}
We have $\tau:\Lambda\rightarrow G^\vee/H\times \B$. Consider the pull back of $L_\lambda$ to the product $G^\vee/H\times \B$ and then using $\tau$ to $\Lambda$, denote the resulting sheaf by $L_\lambda^\Lambda\in K^{G^\vee\times\C^\times}(\Lambda)$.

We can do something similar for every  relevant $G^\vee$ orbit separately. For every $G^\vee$ orbit $\s_j$ on $G^\vee/H\times \B$,  we consider the projection $\tau^{-1}(\s_j)\rightarrow \B$. Pulling back $L_\lambda$ we get an element $L^{j}_\lambda\in K^{G^\vee\times\C^\times}(\tau^{-1}(\s_j))$. Notice that many $\lambda$ will give the same element. 
\end{defn}

\begin{prop}\label{torus action}
Let $\lambda\in X_*(T)=X^*(T^\vee)$.
\begin{enumerate}
    \item The action of $\theta_\lambda$ on $K^{G^\vee\times\C^\times}(\Lambda)$ is given by tensoring with  $L_{\lambda^{-1}}^\Lambda$.
    \item For every $i$, the action of $\theta_\lambda$ preserves $K^{G^\vee\times\C^\times}(\Lambda_i)$.
    \item The action of $\theta_\lambda$ commutes with the map $K^{G^\vee\times\C^\times}(\Lambda_i)\rightarrow K^{G^\vee\times\C^\times}(\Lambda_i\setminus \Lambda_{i-1})$, where the action on $K^{G^\vee\times\C^\times}(\Lambda_i\setminus \Lambda_{i-1})=K^{G^\vee\times\C^\times}(\tau^{-1}(\s_j))$ is given by tensoring with $L^{i}_{\lambda^{-1}}$.
\end{enumerate}
    
\end{prop}

\begin{proof}
    As $L^{St}_\lambda$ is supported on $St_0$ we can compute the convolution with it over 
    $$\{(g_1H,g_2B,g_2B,\phi)|(g_1H,g_2B,\phi)\in \Lambda\}\cong \Lambda$$

    Now, the first item is clear. Tensoring with $L^{\Lambda}_\lambda$ does not change the support and this implies item two.

    The map $K^{G^\vee\times\C^\times}(\Lambda_i)\rightarrow K^{G^\vee\times\C^\times}(\Lambda_i\setminus \Lambda_{i-1})$ is a pullback map and the fact that the pullback of a tensor product is a tensor product of the pullbacks implies the third item. 
\end{proof}

\begin{prop}\label{kernel on zero level}
Let $\s_j$ be some relevant $H$ orbit on $\B$, let $u\in W$ so that $HuB^\vee=\s_j$. 

Let $\lambda_1,\lambda_2\in X^*(T^\vee)$, we have $L^j_{\lambda_1}=L^j_{\lambda_2}$ in $K^{G^\vee}(\Lambda_j\setminus\Lambda_{j-1})$ if and only if $\lambda_1\lambda^{-1}_2$ is trivial on  $u^{-1}(T^\vee\cap H)u$. Moreover, in this case we have $\lambda_1(u^{-1}\rho(q)u)L^j_{\lambda_1}=\lambda_2(u^{-1}\rho(q)u)L^j_{\lambda_2}$ in $K^{G^\vee\times \C^\times}(\Lambda_j\setminus\Lambda_{j-1})$.

\end{prop}

\begin{proof}
 The line bundles $L^{j}_{\lambda_1}$ and $L^{j}_{\lambda_2}$ are equal in $K^{G^\vee\times \C^\times}(\Lambda_j\setminus\Lambda_{j-1})$ if and only if $L_{\lambda_1}$, $L_{\lambda_2}$ are equal on the $H$ orbit $\s_j\subset \B$. The line bundles are equal on $\s_j$ if and only if $L_{\lambda_1\lambda^{-1}_2}$ is trivial on $\s_j$ if and only if the conjugation by $u$ of $\lambda_1\lambda^{-1}_2$ is trivial on  the stabilizer $Stab_H(uB^\vee)=H\cap uB^\vee u^{-1}$ if and only if $\lambda_1\lambda^{-1}_2$ is trivial on $u^{-1}(H\cap T^\vee )u$.

We have $K^{G^\vee}(\Lambda_j\setminus\Lambda_{j-1})=K^{\Tilde{H}}(\s_j)$. We want to find a power of $q$ such that $q^aL_{\lambda_1\lambda^{-1}_2}$ is trivial in $K^{\Tilde{H}}(\s_j)$. The action of $\C^\times\subset \Tilde{H}$ on $uB^\vee$ is given by $z\mapsto \rho(z)uB^\vee=uB^\vee$. So the condition is that $q^a=\lambda_1\lambda_2^{-1}(u^{-1}\rho(q)u)$.
    
\end{proof}

We denote by $\A_{\Lambda_0}\in K^{G^\vee\times \C^\times}(\Lambda)$ the structure sheaf of the irreducible component $\Lambda_0$.

Recall our notation $IM$, for the Iwahori Matsumoto involution (Definition \ref{involution}). Like in Proposition \ref{map BM} we construct a map $S(X)^I\mapsto IM(K^{G^\vee\times \C^\times}(\Lambda))$ by acting on $\A_{\Lambda_0}$ and checking annihilation conditions.

\begin{prop}\label{map KGGm}
        Consider the map $H(G,I)\rightarrow K^{G^\vee\times \mathbb{G}_m}(\Lambda)$ given by acting on $\A_{\Lambda_0}$. It factors through the map $H(G,I)\rightarrow IM(S(X)^I)$ given by twisting the action on $f_{max}$ by the Iwahori Matsumoto involution. 
\end{prop}

Before we can prove Proposition \ref{map KGGm} we need to recall how the elements of the from $T_s+1$ for $s$ a simple reflection act (for more details see Section 7.6 of \cite{Chriss1997RepresentationTA}). Let $St_s\subset St$ be the irreducible component of $St$ corresponding to $s$. $St_s$ sits above the closure of the orbit $\s_s\subset \B\times \B$. We have the first projection $\overline{\s_s}\rightarrow \B$ and we have the sheaf of relative one forms $\Omega^1_{\overline{\s_s}/\B}=Coker(\Omega^1_{\B}\rightarrow\Omega^1_{\overline{\s_s}})$. We denote this sheaf by $\Omega^1_s$. 

We define $Q_s=\tau^*\Omega^1_{\overline{\s_s}/\B}\in K^{G^\vee\times \C^\times}(St_s)$ and extend it to $K^{G^\vee\times \C^\times}(St)$. 

The element $T_s+1\in H(G,I)$ acts by convolution with $-qQ_s$. 

\begin{proof}
    In light of Proposition \ref{ideal generaotrs} it is enough to show that the elements of the form:

    \begin{enumerate}
        \item $IM(T_{s_{2i-1}}-q)=-\frac{1}{q}(T_{s_{2i-1}}+1)$
        \item $IM(T_{s_{2i-1}}T_{s_{2i}}-T_{s_{2i+1}}T_{s_{2i}})=\frac{1}{q}(T_{s_{2i+1}}-T_{s_{2i-1}})IM(T_{s_{2i}})$
        \item $IM(\theta_{\lambda})-q^\frac{l(\lambda^-)-l(\lambda^+)}{2}=\theta_{\lambda^{-1}}-q^\frac{l(\lambda^-)-l(\lambda^+)}{2}$ for $\lambda\in L_{max}$.
    \end{enumerate}

    act as zero on $\A_{\Lambda_0}$.

    The result for elements of the third type follows from Propositions \ref{torus action} and \ref{kernel on zero level}.

    Let $s=s_{2i-1}$, like in the proof of Proposition \ref{map BM} we get that $T_s\A_{\Lambda_0}$ is supported on $\Lambda_0$. Moreover, to compute the action of $T_s+1$ on $\A_{\Lambda_0}$, we compute  $\pi_{2*}(\pi^*_1\A_{\Lambda_0}\otimes \pi_3^*Q_s)$ for the sheaf $Q_s$ supported on $St_s$. The supports of $\pi^*_1\A_{\Lambda_0}$ and $\pi_3^*Q_s$ do not intersect, so  $T_s+1$ acts as $0$ as desired.

    We need to show that the elements of the form $T_{s_{2i-1}}IM(T_{s_{2i}})-T_{s_{2i+1}}IM(T_{s_{2i}})$ act as zero on $\A_{\Lambda_0}$. 
    
    It is same as showing that the actions of $(T_{s_{2i+1}}+1)(T_{s_{2i}}+1)$ and $(T_{s_{2i-1}}+1)(T_{s_{2i}}+1)$ agree because we already know that $T_{s_{2i-1}}$ and $T_{s_{2i+1}}$ act the same way on $\A_{\Lambda_0}$.

    For simplicity we assume that $n=2$ and $i=1$.

    The action of $(T_{s_{3}}+1)(T_{s_{2}}+1)$ is by convolution with a certain sheaf supported on $St_{s_{3}s_{2}}$, denote this sheaf by $P_3\in K^{G^\vee\times\C^\times}(St)$, similarly we also have $P_1\in K^{G^\vee\times\C^\times}(St)$. This sheaves are in fact pulled back from $K^{G^\vee\times\C^\times}(\B\times \B)$. Let $R_1,R_3$ denote the sheaves in $K^{G^\vee\times\C^\times}(\B\times \B)$ whose pullbacks are $P_1,P_3$ respectively.  

    Recall we have a filtration on $\Lambda$. 
    
    We have $\Lambda_0\subset \Lambda_1\subset \Lambda_2=\Lambda$ each time adding a single irreducible component according to the Burhat order on relevant $H$ orbits on $\B$, $\s_0<\s_1<\s_2$ . 
    
    Let $\Lambda'=\Lambda\setminus\Lambda_1$ and let $j:\Lambda'\rightarrow\Lambda$ be the open embedding. We have the pullback map $j^*:K^{G^\vee\times\C^\times}(\Lambda)\rightarrow K^{G^\vee\times\C^\times}(\Lambda')$. We begin with checking that $j^*((T_{s_{3}}+1)(T_{s_{2}}+1)\A_{\Lambda_0})=j^*((T_{s_{1}}+1)(T_{s_{2}}+1)\A_{\Lambda_0})$.

    By definition $j^*((T_{s_{3}}+1)(T_{s_{2}}+1)\A_{\Lambda_0})=j^*(\pi_{2*}(\pi_1^*\A_{\Lambda_0}\otimes \pi_3^*P_3))$. The map $\pi_2$ is proper so we may apply the base change theorem. Consider the following commutative diagram:

    \[
    \begin{tikzcd}
\Lambda'\times_{\mathfrak{g}^{\vee*}} T^*\B 
    \arrow[r,"j'"] 
    \arrow[d,"\pi_2'"] &
\Lambda\times_{\mathfrak{g}^{\vee*}} T^*\B
    \arrow[d,"\pi_2"] \\
\Lambda'
    \arrow[r, "j"] 
& \Lambda
\end{tikzcd}
\]

We have $j^*(\pi_{2*}(\pi_1^*\A_{\Lambda_0}\otimes \pi_3^*P_3))=\pi_{2*}'(j'^*(\pi_1^*\A_{\Lambda_0}\otimes \pi_3^*P_3))=\pi_{2*}'((\pi_1\circ j')^*\A_{\Lambda_0}\otimes (\pi_3\circ j')^*P_3))$.

The sheaf $(\pi_1\circ j')^*\A_{\Lambda_0}$ is equal to the structure sheaf of $$Z=\{(g_1H,g_2B^\vee,g_3B^\vee,\phi)|Hg_1^{-1}g_2B^\vee=\s_1,Hg_1^{-1}g_3B^\vee=\s_3,(g^{-1}_1\phi)|_\mathfrak{h}=\psi,(g^{-1}_2\phi)|_\mathfrak{b^\vee}=(g^{-1}_3\phi)|_\mathfrak{b^\vee}=0\}$$

The conditions $(g^{-1}_1\phi)|_\mathfrak{h}=\psi$ and $(g^{-1}_3\phi)|_\mathfrak{b^\vee}=0$ for $g_1,g_2\in G$ such that $Hg_1^{-1}g_3B^\vee=\s_3$ force $\phi=g_1\psi$.  

We have the projection $Z\rightarrow St$ which we denote by $\pi_3'$, the sheaf $(\pi_1\circ j')^*\A_{\Lambda_0}\otimes (\pi_3\circ j')^*P_3)$ is equal to the extension by zero of $\pi_3'^*P_3$.

Let $Z_H=\{(g_2B^\vee,g_3B^\vee)|Hg_2B^\vee=\s_1,Hg_3B^\vee=\s_3\}=\s_1\times\s_3$, we have 

$Z=(G^\vee\times\C)^\times\times^{\Tilde{H}}Z_H$ and $K^{G^\vee\times\C^\times}(Z)\cong K^{\Tilde{H}}(Z_H)$.

We denote by $k$ the embedding $k:Z_H\rightarrow \B\times \B$. 

Under the isomorphism $K^{G^\vee\times\C^\times}(Z)\cong K^{\Tilde{H}}(Z_H)$ the sheaf $\pi_3'^*P_3$ corresponds to $k^*R'_3$. Here $R'_3\in K^{\Tilde{H}}(\B\times \B)$ is the restriction of $R_3\in K^{G^\vee\times\C^\times}(\B\times \B)$. The sheaf $R_1'$ is defined similarly. Let $pr_2:Z_H\rightarrow \s_3$ be the projection on the second coordinate. It is enough to show that $pr_{2*}k^*R_3=pr_{2*}k^*R_1$ in $K^{\Tilde{H}}(\s_3)$. This follows from symmetry of $H$ and $\s_3$ with respect to switching the simple reflections $s_1$ and $s_3$.

The upshot is that $j^*((T_{s_{3}}+1)(T_{s_{2}}+1)\A_{\Lambda_0})=j^*((T_{s_{1}}+1)(T_{s_{2}}+1)\A_{\Lambda_0})$. By the short exact sequence $0\rightarrow K^{G^\vee\times\C^\times}(\Lambda_2)\rightarrow K^{G^\vee\times\C^\times}(\Lambda)\rightarrow K^{G^\vee\times\C^\times}(\Lambda\setminus \Lambda_2)\rightarrow 0$ we get that $((T_{s_{3}}+1)(T_{s_{2}}+1)-(T_{s_{1}}+1)(T_{s_{2}}+1))\A_{\Lambda_0}\in K^{G^\vee\times\C^\times}(\Lambda_2)$. 

Next, we check that the action of $\theta_\lambda\in \Theta$ for $\lambda$ of the form $\lambda(t)=diag(t^a,t^b,t^{-b},t^{-a})$ on $((T_{s_{3}}+1)(T_{s_{2}}+1)-(T_{s_{1}}+1)(T_{s_{2}}+1))\A_{\Lambda_0}$ is by multiplication by an element of $\C[q^{\pm 1}]$. From Propositions \ref{torus action} and \ref{kernel on zero level} it follows that any element with this property in  $K^{G^\vee\times\C^\times}(\Lambda_2)$ is zero. This will finish the proof.

The first step is a calculation done using Theorem \ref{Ber_presentation}.

For such $\lambda$ we have $(\theta_{s_1(\lambda)}-\theta_{s_3(\lambda)})\A_{\Lambda_0}=0$. This holds as $s_3(\lambda)s_1(\lambda)\in L_{max}$ and we can write $s_3(\lambda)s_1(\lambda^{-1})=\lambda^+(\lambda^-)^{-1}$ for $\lambda^+,\lambda^-$ dominant with $l(\lambda^+)=l(\lambda^-)$. We use the result of the proposition for elements of the third type to get that $$(\theta_{s_1(\lambda)}-\theta_{s_3(\lambda)})\A_{\Lambda_0}=\theta_{s_1(\lambda)}(1-\theta_{s_1(\lambda^{-1})s_3(\lambda)})\A_{\Lambda_0}=0$$

Similarly, $\theta_{s_2(s_1(\lambda))}\A_{\Lambda_0}=\theta_{s_2(s_3(\lambda))}\A_{\Lambda_0}=c(\lambda)\A_{\Lambda_0}$ for some $c(\lambda)\in \C[q^{\pm1}]$. 

Now we can do the computation. The element $\theta_\lambda T_{s_3}T_{s_{2}}$ is equal to:

$$(T_{s_3}\theta_{s_{3}(\lambda)}+(q-1)\frac{\theta_\lambda-\theta_{s_{3}(\lambda)}}{1-\theta_{-\alpha_{3}^\vee}})T_{s_{2}}=T_{s_3}T_{s_{2}}\theta_{s_{2}(s_{3}(\lambda))}+(q-1)T_{s_{3}}\frac{\theta_{s_{3}(\lambda)}-\theta_{s_{2}(s_{3}(\lambda))}}{1-\theta_{-\alpha_{2}^\vee}}+(q-1)\frac{\theta_\lambda-\theta_{s_{3}(\lambda)}}{1-\theta_{-\alpha_{3}^\vee}})T_{s_{2}}$$ 

The elements $(q-1)T_{s_{3}}\frac{\theta_{s_{3}(\lambda)}-\theta_{s_{2}(s_{3}(\lambda))}}{1-\theta_{-\alpha_{2}^\vee}}$ and $(q-1)T_{s_{1}}\frac{\theta_{s_{1}(\lambda)}-\theta_{s_{2}(s_{1}(\lambda))}}{1-\theta_{-\alpha_{2}^\vee}}$ act the same way on $\A_{\Lambda_0}$ so they cancel out. 

We have $(T_{s_3}T_{s_{2}}\theta_{s_{2}(s_{3}(\lambda))}-T_{1}T_{s_{2}}\theta_{s_{2}(s_{1}(\lambda))})\A_{\Lambda_0}=c(\lambda)(T_{s_3}T_{s_{2}}-T_{s_1}T_{s_{2}})\A_{\Lambda_0}$ so we only need to understand the last term.

Denote $m_3=\theta_{-\alpha_{3}^\vee}$ and $s_2(m_3)=\theta_{-s_2(\alpha_{3}^\vee)}$. Similarly,  $m_1=\theta_{-\alpha_{1}^\vee}$ and $s_2(m_1)=\theta_{-s_2(\alpha_{1}^\vee)}$. We have
 $\frac{\theta_\lambda-\theta_{s_{3}(\lambda)}}{1-\theta_{-\alpha_{3}^\vee}}=\sum_{i=0}^{\alpha_3(\lambda)}\theta_\lambda m_3^i$.

 Therefore $\frac{\theta_\lambda-\theta_{s_{3}(\lambda)}}{1-\theta_{-\alpha_{3}^\vee}}T_2=\sum_{i=0}^{\alpha_3(\lambda)}\theta_\lambda m_3^iT_2=\theta_\lambda(\sum_{i=0}^{\alpha_3(\lambda)}T_2s_2(m_3)^i+(q-1)\frac{m_3^i-s_2(m_3)^i}{1-\theta_{-\alpha_2^\vee}})$.

 Notice that $m_3$ and $m_1$ act the same on $\A_{\Lambda_0}$. Also $s_2(m_3)$ and $s_2(m_1)$ act the same on $\A_{\Lambda_0}$. Using that we conclude that for $c(\lambda)\in \C[q^{\pm 1}]$.
 
 $$\theta_\lambda((T_{s_{3}}+1)(T_{s_{2}}+1)-(T_{s_{1}}+1)(T_{s_{2}}+1))\A_{\Lambda_0}=c(\lambda)((T_{s_{3}}+1)(T_{s_{2}}+1)-(T_{s_{1}}+1)(T_{s_{2}}+1))\A_{\Lambda_0}$$

\end{proof}

We denote the obtained map by $\Phi_X$. We have to show that it is injective and surjective.

\begin{prop}
    The map $\Phi_X$ defined in Proposition \ref{map KGGm} is an isomorphism of $H(G,I)$ modules. 
\end{prop}

\begin{proof}
    We begin with checking injectivity. 
    
    First we check injectivity of the map $\C[I\backslash X]\rightarrow K^{G^\vee}(\Lambda)\otimes sgn_f$ obtained from $\Phi_X$ by setting $q=1$. Let $a\in \C[I\backslash X]$ be in the kernel. We can find $h\in \C[W_{aff}]$ such that $a=h\mathcal{U}_{\max}$. Write $h=\sum _{w\in W,t\in T/T^0}a_w^ttw$ for some coefficients $a_w^t\in \C$.  Consider the maximal $i$ such that there are $w\in W$ and $t\in T/T^0$ such that $a_w^t\neq 0$ and $\Phi_X(w\mathcal{U}_{max})\in K^{G^\vee}(\Lambda_i)$. Let $x_0\in B\backslash X$ be the orbit $x_0=w\mathcal{U}_{max}$. It is the orbit corresponding to the irreducible component of $\Lambda$ in $\Lambda^i\setminus\Lambda_{i-1}$ under the bijection constructed in Section \ref{s3}.

    Let $j^*:K^{G^\vee}(\Lambda_i)\rightarrow K^{G^\vee}(\Lambda_i\setminus\Lambda_{i-1})$
    
    Consider $j^*\Phi_X(a)$ in $K^{G^\vee}(\Lambda_i\setminus\Lambda_{i-1})$. By Propositions \ref{torus action} and \ref{kernel on zero level} we know that the action of $T/T^0$ on $K^{G^\vee}(\Lambda_i\setminus\Lambda_{i-1})$ factors through $T/(T^0Stab_T(x_0))$. Here, $Stab_T(x_0)$ is the stabilizer of $x_0$ in $T/T^0$. 

    Let $W_{x_0}=\{w\in W|w\mathcal{U}_{max}=x_0\}$. Notice that for any other $w'\notin W_{x_0}$ we have that for every $t\in T/T^0$, $j^*\Phi_X(a^t_{w'}tw'\mathcal{U}_{max})=0$.  

    We claim that $\sum_{w\in W_{x_0},t\in T/T^0}a_{w}^ttw\mathcal{U}_{max}=\sum_{w\in W_{x_0},t\in T/T^0}a_{w}^ttx_0=0$.
    
    This holds because the action of $T/T^0$ on $x_0$ has the same stabilizer as the action on $j^*\Phi_X(x_0)$ and $j^*\Phi_X(\sum_{w\in W_{x_0},t\in T/T^0}a_{w}^ttx_0)=j^*\Phi_X(a)=0$.
    
     Continuing like this for smaller values of $i$ shows that $a=0$.
    
    Now we return to the original map, for generic $q$. Assume that we have some $a\in S(X)^I$ with $\Phi_X(a)=0$. 
    
    We can write $a$ as a linear combination of characteristic functions of $I$ orbits on $X$ over $\C[q,q^{-1}]$, i.e. $a=\sum _j a_j1_{x_j}$ with $a_j\in \C[q,q^{-1}]$.
    
    We can multiply $a$ by powers of $q$ such that all $a_j$ become polynomial in $q$. We know that $\Phi_X$ is injective when we specialize to $q-1$. Therefore $q-1|a_j(q)$ for every $j$ and then we can divide everything by $q-1$ and still get an element in the kernel of $\Phi_X$. We can continue this way, reducing the degrees of $a_j$ until we get that $a_j=0$ for every $j$ which implies that $a=0$ and $\Phi_X$ is injective.

    In order to prove surjectivity it is enough to show that $\A_{\Lambda_0}$ generates $K^{G\times \C^\times}(\Lambda)$ over $K^{G\times \C^\times}(St)$.

    It is enough to show that for every irreducible component $\Lambda_i$, we can get $\A_{\Lambda_i}$ from $\A_{\Lambda_0}$. We do it by induction on the Bruhat order and we use the cellular fibration lemma. Assume we know this for all irreducible components $\Lambda_j$ such that $\s_j<_{Bruhat}\s_i$, then we can choose some simple reflection $s\in W$ and some orbit $\s_j$ smaller then $\s_i$ with $s\s_j=\s_i$. 
    
    Then $\A_{St_s}*\A_{\Lambda_j}$ is supported on irreducible components which correspond to orbits which are smaller or equal to $\s_i$. Over $\Lambda_i$ we get $\A_{\Lambda_i}$. By induction this finishes the proof.
\end{proof}

\begin{cor}
    Theorem \ref{relative KL} holds.
\end{cor}
\end{section}

\begin{section}{Fixed points and central characters}\label{s8}

In this Section we will use Theorem \ref{relative KL} to deduce results about $X$ distinguished irreducible representations generated by their $I$ fixed vectors (recall Definition \ref{dist rep}).

For a representation $\pi$ denote by $\pi^\vee$ its contragredient.
 By Frobenius reciprocity, an irreducible representation $\pi$ is $X$ distinguished if and only if $Hom_G(S(X),\pi^\vee)\neq 0$.  
 
 For representations generated by their $I$ fixed vectors this condition is equivalent to 
 
 $Hom_{H(G,I)}(S(X)^I,(\pi^\vee)^I)\neq 0$ (see \cite{Borel1976}). Therefore, we are interested in simple $H(G,I)$ modules which are irreducible quotients of $S(X)^I$.

We know that every simple $H(G,I)$ module is parametrized by a Deligne Langlands parameter. Meaning a pair $(t,n)\in G^\vee\times \mathfrak{g^\vee}$ with $t$ semi simple and $n$ nilpotent such that $t^{-1}nt=q_rn$. Recall that $q_r$ is the size of our residue field. Denote by $v=\sqrt{q_r}$ the positive square root of $q_r$.

Fix $t\in G^\vee$ semisimple and consider $a=(t,v)\in G^\vee\times\C^\times$. 

\begin{Remark}
    We need the square root to make our $\C^\times$ action consists with the one used in \cite{Chriss1997RepresentationTA} and \cite{Kazhdan1987ProofOT} (see Remark \ref{linear action}).
\end{Remark}

The element $a$ determines a central character of $H(G,I)$ as $Z(H(G,I))\cong R(G^\vee\times\C^\times)$ (see Section 8.1 of \cite{Chriss1997RepresentationTA}). We denote by $\C_a$ the one dimensional representation of $Z(H(G,I))$ corresponding to $a$. Denote $H_a=H(G,I)\otimes_{Z(H(G,I))}\C_a$.

For a variety $Z$ with an action of $a$ we denote by $Z^a$ the fixed points.

\begin{prop}[8.1.5 of \cite{Chriss1997RepresentationTA}]\label{fixed points algebra}

We have an algebra isomorphism $H_a\cong H^{BM}_\bullet(St^a)$.
    
\end{prop}

We also have a relative version. First we need to determine the values of $a$ such that $\Lambda^a\neq\emptyset$.

\begin{prop}\label{non_empty_fixed_points}
    Let $a=(t,v)\in G^\vee\times\C^\times$ with $t$ semisimple. $\Lambda^a\neq\emptyset$ if and only if $t$ is conjugate to an element of the form $diag(t_1v^{-1},...,t_nv^{-1},t_1v,...,t_nv)$.
\end{prop}

\begin{proof}
    Let $(g_1H,g_2B^\vee,\phi)\in\Lambda$ be a fixed point of $a$. 
    
    Then we have $tg_1\rho(v)H=g_1H$, $tg_2B^\vee=g_2B^\vee$ and $t^{-1}q_r\phi=\phi$. Acting by $G^\vee$ we can assume that $g_1=1$ and that $g_2=w\in W$. Assuming that, we have $t\rho(v)H=H$ and $twB^\vee=wB^\vee$. Equivalently, $t\rho(v)\in H$ and $t\in wB^\vee w^{-1}$.   

    Using the condition $t\rho(v)\in H$ we can conjugate $t$ by an element of $H$ to be of the form $diag(t_1v^{-1},...,t_nv^{-1},t_1v,...,t_nv)$.

    For the other direction, assume that $t$ has the form $t=diag(t_1v^{-1},...,t_nv^{-1},t_1v,...,t_nv)$, let $w\in W$ be such that $HwB^\vee$ is a relevant orbit. We have $(H,wB^\vee,\psi)\in \Lambda$ a fixed point of $a=(t,v)$.
\end{proof}

As a corollary we obtain the following.

\begin{cor}\label{central cahrachter}
    Assume that $\pi$ is an irreducible representation of $G$ generated by its $I$ fixed vectors, with Deligne Langlands parameter $(t,n)\in G^\vee\times\mathfrak{g}^\vee$. If $\pi$ is $X$ distinguished then $t$ can be conjugated to an element of the form  $diag(t_1v^{-1},...,t_nv^{-1},t_1v,...,t_nv)$.
\end{cor}

\begin{proof}
    Let $a=(t,v)$ and let $\C_a$ be the one dimensional representation of $Z(H(G,I))$ corresponding to $a$. 

    The center $Z(H(G,I))$ acts on $IM((\pi^\vee)^I)$ by the character defined by $a$. This holds because both passing to the contragrident and applying the Iwahori Matsumoto involution inverses the central character. Therefore $IM((\pi^\vee)^I)=IM((\pi^\vee)^I)\otimes_{Z(H(G,I))}\C_a$ is a quotient of $IM(S(X)^I)\otimes_{Z(H(G,I))} \C_a$. In particular $IM(S(X)^I)\otimes_{Z(H(G,I))} \C_a\neq 0$.

    On the other hand, we have $$IM(S(X)^I)\otimes_{Z(H(G,I))} \C_a\cong K^{G^\vee\times \C^\times}(\Lambda)\otimes_{Z(H(G,I))} \C_a\cong  K^{G^\vee\times \C^\times}(\Lambda)\otimes_{R(G^\vee\times\mathbb{G}_m)} \C_a$$

    By the localization theorem (Proposition 4.1 of \cite{Segal1968EquivariantKTheory}) $\Lambda^a=\emptyset$ implies $$K^{G\times \C^\times}(\Lambda)\otimes_{R(G\times\mathbb{G}_m)} \C_a=0$$ 

    Thus, the result follows from Proposition \ref{non_empty_fixed_points}.
\end{proof}

\begin{Remark}
    In our case the dual group $G_X^\vee$ defined in \cite{Knop2017TheDG} is $GL_n$ and the map $G_X^\vee\times SL_2\rightarrow G_X^\vee$ constructed in \cite{Knop2017TheDG} is the following $g,\begin{pmatrix}
        a & b\\ c& d
    \end{pmatrix}\mapsto \begin{pmatrix}
        ag & bg\\ cg & dg
    \end{pmatrix}$. 

    Corollary \ref{central cahrachter} is equivalent to saying that if $\pi$ is $X$ distinguished with Deligne Langalnds parameter $(t,n)$ then $t$ can be conjugated into the image of $G_X^\vee\times\begin{pmatrix}
        v^{-1} &0\\0 & v
    \end{pmatrix}\subset G_X^\vee\times SL_2$. 
\end{Remark}

Now, we will give a relative version of Proposition \ref{fixed points algebra} under the assumption that $\Lambda^a\neq\emptyset$.

By the convolution construction we get an action of $H^{BM}_\bullet(St^a)$ on $H^{BM}_\bullet(\Lambda^a)$. 

Recall that $\Lambda=G^\vee\times^H \Lambda_H$ and that $\Lambda_H$ satisfies the cellular fibration lemma over $\B_n$. 

Also, recall the definition of $\Tilde{H}$ from Proposition \ref{semi direct product} and its action on $\Lambda_H$. 

We have the embedding $\Tilde{H}\subset G\times \mathbb{G}_m$. Notice that by Proposition \ref{non_empty_fixed_points} $\Lambda^a\neq\emptyset$ if and only if $a$ can be conjugated into $\Tilde{H}$.

Let $S=\{\phi\in \mathfrak{g}^{\vee*},\phi|_\mathfrak{h}=\psi\}$, clearly $H$ acts on $S$. Notice that $\Lambda_H=T^*\B\times_{\mathfrak{g}^{\vee*}}S$. Thus we can form a convolution diagram and define an action of $H^{BM}_\bullet(St^a)$ on $H^{BM}_\bullet(\Lambda_H^a)$.

\begin{prop}\label{fixed points module}
    We have an isomorphism $H^{BM}_\bullet(\Lambda_H^a)\cong IM(S(X)^I)\otimes_{Z(H(G,I))} \C_a$ as $H_a$ modules compatible with the isomorphism given in Proposition \ref{fixed points algebra}.
\end{prop}

\begin{proof}
    We follow the same ideas as in the proof of Proposition \ref{fixed points algebra}. 

    Let $A\subset \Tilde{H}$ be the minimal closed subgroup that contains $a$. 
    
    We have   $$IM(S(X)^I)\otimes_{Z(H(G,I))} \C_a\cong K^{G\times \C^\times}(\Lambda)\otimes_{R(G\times \C^\times)} \C_a\cong K^{\Tilde{H}}(\Lambda_H)\otimes_{R(G\times \C^\times)} \C_a\cong K^{\Tilde{H}}(\Lambda_H)\otimes_{R(\Tilde{H})} \C_a$$ 
    
    By part 3 of the cellular fibration lemma we have $K^{\Tilde{H}}(\Lambda_H)\otimes_{R(\Tilde{H})}R(A)\cong K^A(\Lambda_H)$. Tensoring with $\C_a$ over $R(A)$ we get $K^{\Tilde{H}}(\Lambda_H)\otimes_{R(\Tilde{H})} \C_a\cong K^A(\Lambda_H)\otimes_{R(A)}\C_a$.

    By Lemma 5.11.5 of \cite{Chriss1997RepresentationTA} we have $K^A(\Lambda_H)\otimes_{R(A)}\C_a\cong K(\Lambda_H^A)$. We can find an isomorphism which is compatible with the convolution action using Lemma 5.11.10 of \cite{Chriss1997RepresentationTA}.

    Using the cellular fibration lemma and Theorem 5.9.19 of \cite{Chriss1997RepresentationTA} we have $K(\Lambda_H^A)\cong H_\bullet^{BM}(\Lambda_H^A)$. We can take the isomorphism to be compatible with the convolution action by Theorem 5.11.11 of \cite{Chriss1997RepresentationTA}.

    By construction of $A$ we have $H_\bullet^{BM}(\Lambda_H^A)=H_\bullet^{BM}(\Lambda_H^a)$ which completes the proof. 
\end{proof}

\end{section}

\begin{section}{Sheaf theoretic analysis of distinguished representations}\label{s9}

In this Section we continue the analysis of $S(X)^I\otimes_{Z(H(G,I))} \C_a$  that we started in the previous section. This time we use a sheaf theoretic description of the module $H^{BM}_\bullet(\Lambda_H^a)$. We follow the ideas of Section 8.6 of \cite{Chriss1997RepresentationTA}.

Fix $a=(t,v)\in \Tilde{H}\subset G\times\mathbb{G}_m$ such that $\Lambda_H^a\neq\emptyset$.

Recall the definition $S=\{\phi\in \mathfrak{g}^{\vee*},\phi|_\mathfrak{h}=\psi\}$ from the previous section.

Let $D_1,D_2$ be the dualizing complexes on $S^a$ and on $(T^*\B)^a$ respectively. Both $S^a$ and $(T^*\B)^a$ are smooth so the dualizing complexes are just constant sheaves shifted by twice the dimension of the corresponding variety.

We denote the moment map of $T^*\B$ by $\pi$ and the embedding of $S$ by $\mu$. We denote by $\pi$ and $\mu$ also the restrictions to the fixed points of $a$.

\begin{prop}\label{pass to Ext}

We have a commutative diagram whose vertical lines are isomorphisms.

    \[
    \begin{tikzcd}
H^{BM}_{\bullet}(St^a) \otimes H^{BM}_{\bullet}(\Lambda_H^a) 
    \arrow[r, "\textit{convolution action}"] 
    \arrow[d]
& H^{BM}_{\bullet}(\Lambda_H^a) 
    \arrow[d] \\
\text{Ext}^\bullet(\pi_*D_2, \pi_*D_2) \otimes \text{Ext}^\bullet(\pi_*D_2, \mu_*D_1) 
    \arrow[r, "\textit{composition}"] 
& \text{Ext}^\bullet(\pi_*D_2, \mu_*D_1)
\end{tikzcd}
\]

The Ext groups are computed in the category ${D_c^b((\mathfrak{g}^{\vee*})^a)}$ of bounded derived constructible sheaves on  $(\mathfrak{g}^{\vee*})^a$.

\end{prop}

\begin{proof}
Notice that the map $\pi$ is proper and $\mu$ is a closed embedding so it is also proper. Thus, the proposition is a special case of Proposition 8.6.16 of \cite{Chriss1997RepresentationTA}.     
\end{proof}

Using a Killing form we pass from $\mathfrak{g}^{\vee*}$ to $\mathfrak{g}$. The moment map $\pi$ from $T^*\B$ to $\mathfrak{g}^{\vee*}$ becomes the Springer resolution of the nilpotent cone $\Tilde{N}\rightarrow N$. We identify $S$ with the subspace of $\mathfrak{g}$ given by $S=\{\begin{pmatrix}
        a & I_n  \\
        b & -a
        \end{pmatrix}|a,b\in \mathfrak{gl}_n\}$ and we denote by $\mu$ the embedding $\mu:S\hookrightarrow \mathfrak{g}$.

We can also replace the sheaves $D_1,D_2$ with the constant sheaves $\C_{S^a}$ and $\C_{\Tilde{N}^a}$. 
Using this notations, we study the $Ext^\bullet_{\mathfrak{g}^a}(\pi_*\C_{\Tilde{N}^a},\pi_*\C_{\Tilde{N}^a})$ module $Ext^\bullet_{\mathfrak{g}^a}(\pi_*\C_{\Tilde{N}^a},\mu_*\C_{S^a})$.

Denote by $C_{G^\vee}(t)\subset G^\vee$ the centralizer of $t$. Denote by $\mathcal{E}$ the collection of $C_{G^\vee}(t)$ orbits on $\mathfrak{g}^a$. For any $f\in \mathcal{E}$ we denote by $IC_f$ the IC sheaf corresponding to $f$.

Now we recall the relation between simple $H^{BM}_\bullet(St^a)$ modules and nilpotent elements $n$ such that $t^{-1}nt=q_rn$. These nilpotent elements are precisely the elements of $\mathfrak{g}^a$. See Section 8.6 of \cite{Chriss1997RepresentationTA} for more details.

By the decomposition theorem (Theorem 8.4.8 in \cite{Chriss1997RepresentationTA}) we have $\pi_*\C_{\Tilde{N}^a}=\bigoplus_{f\in \mathcal{E}}L_f\otimes IC_f$. Where $L_f$ is some graded vector space. By a result of Kazhdan and Lusztig (see \cite{Kazhdan1987ProofOT}), we have $L_f\neq 0$ for every $f\in\mathcal{E}$.

Denote $A=\bigoplus_k Ext^k_{\mathfrak{g}^a}(\pi_*\C_{\Tilde{N}^a},\pi_*\C_{\Tilde{N}^a})$.

We have $$A=\bigoplus_{k\in \Z,c,d\in\mathcal{E}}Ext^k(IC_c,IC_d)\otimes Hom(L_c,L_d)=\bigoplus_{c\in \mathcal{E}}End(L_c)\oplus\bigoplus_{k>0,c,d\in\mathcal{E}}Ext^k(IC_c,IC_d)\otimes Hom(L_c,L_d)$$ 

The radical of $A$ is $Rad(A)=\bigoplus_{k>0,c,d\in\mathcal{E}}Ext^k(IC_c,IC_d)\otimes Hom(L_c,L_d)$ and the semi-simple part is $A/Rad(A)=\bigoplus_{c\in \mathcal{E}}End(L_c)$. To each orbit $c\in \mathcal{E}$ corresponds a simple $A$ module $L_c$. 

We go back to our problem.

We have an $A$ module $V$ given by: $$V:=\bigoplus_k Ext^k_{\mathfrak{g}^a}(\pi_*\C_{\Tilde{N}^a},\mu_*\C_{S^a})=\bigoplus_{k\in \Z,c\in\mathcal{E}}Ext^k(IC_c,\mu_*\C_{S^a})\otimes L_c^*$$

We are interested in simple modules of $A$ with are quotients of $V$. 

We have the following result.
\begin{prop}\label{sheaf condition}
    Let $\pi$ be an irreducible representation of $G$, generated by its $I$ fixed vectors. Let $(t,n)$ be a Deligne Langlands parameter of $\pi$. Let $c\in\mathcal{E}$ be the orbit of $n\in \mathfrak{g}^a$.

    The following conditions are equivalent:

    \begin{enumerate}
        \item Let $Z(\pi)$ be the Zelevinsky dual of $\pi$ (see Section 9 of \cite{Zelevinsky1980}) and let $Z(\pi)^\vee$ be its contragrident representation. The representation $Z(\pi)^\vee$ is $X$ distinguished. 
        \item The $H(G,I)$ module  $IM(\pi^I)$ is a quotient of $S(X)^I$.
        \item The simple $A$ module $L_c$ is a quotient of $V$.
        \item The natural map $\bigoplus_{k>0,d\in\mathcal{E}}Ext^k(IC_c,IC_d)\otimes \bigoplus Ext^{\bullet}(IC_d,\mu_*\C_{S^a})\rightarrow\bigoplus Ext^\bullet(IC_c,\mu_*\C_{S^a})$ is not surjective. We denote this map by $Comp_c$.
    \end{enumerate}
    
\end{prop}

\begin{proof}
    To see the equivalence between items 1 and 2 it is enough to show that $Z(\pi)^I\cong IM(\pi^I)$. This is well known, see Theorem 2 in \cite{Duality_for_representations_of_a_Hecke_algebra}.    

    We already gave the argument for the equivalence of items 2 and 3. We prove that items 3 and 4 are equivalent.

    To prove that 3 implies 4, assume that $Comp_c$  is surjecitve.

    Using the fact that $L_d\neq 0$ for every $d\in \mathcal{E}$ we get that $\bigoplus_{k\in \Z}Ext^k(IC_c,\mu_*\C_{S^a})\otimes L_c^*$ is in the image of $Rad(A)\otimes V\rightarrow V$. 

    Let $\varphi:V\rightarrow L_c$ be a surjective map of $A$ modules. The radical $Rad(A)$ acts as zero on $L_c$ so $\varphi(Rad(A)V)=0$. In particular $\varphi$ is zero on $\bigoplus_{k\in \Z}Ext^k(IC_c,\mu_*\C_{S^a})\otimes L_c^*$. The algebra $End(L_c)\subset A$ acts trivially on $\bigoplus_{k\in \Z,d\in \mathcal{E},d\neq c}Ext^k(IC_d,\mu_*\C_{S^a})\otimes L_d^*$, the only point in $L_c$ fixed under the action of $End(L_c)$ is 0. 
    
    Thus, $\varphi$ must be zero on $\bigoplus_{k\in \Z,d\in \mathcal{E},d\neq c}Ext^k(IC_d,\mu_*\C_{S^a})\otimes L_d^*$. We got that $\varphi=0$.

    For the opposite direction, if $Comp_c$ is not surjecitve, we get a copy of $L_c^*$ in $V/Rad(A)V$ and it gives us a non trivial map of $A/Rad(A)$ modules $V/Rad(A)V\rightarrow L_c$.
\end{proof}

We give a general necessary condition on $c$ such that the equivalent conditions of Proposition \ref{sheaf condition} hold.

First, let us introduce some notations.

\begin{Not}
    We denote $\mathfrak{g}^{*a}$ by $Q$ and we abuse the notation and denote $S^a$ by $S$. The element $a$ is fixed so there is no harm in committing it for the notation.

    Denote by $\mathcal{E}_S$ the collection of orbits which intersect $S$. Let $Q_S$ be the union of all orbits in $\mathcal{E}$ whose closure intersects $S$. Let $j_S:Q_s\rightarrow Q$ be the open embedding.
\end{Not}

 We first restrict to $Q_S$ using the next proposition.

\begin{prop}
   for any $c,d\in\mathcal{E}$ we have  $$Ext^{\bullet>0}(IC_c,IC_d)\otimes Ext^\bullet(IC_d,\mu_*\C)\cong Ext^{\bullet>0}(j_S^*IC_c,j_S^*IC_d)\otimes Ext^\bullet(IC_d,\mu_*\C)$$
\end{prop}

\begin{proof}
    We denote by $hom(\mathcal{F},\mathcal{G})$ the inner hom of two sheaves $\mathcal{F},\mathcal{G}$.

    Notice that $Ext^{\bullet}(j_S^*IC_c,j_S^*IC_d)=H^{\bullet}(j_S^*hom(IC_c,IC_d))$ because $j_S^*=j_S^!$ as $j_S$ is an open embedding.

    We know that $hom(IC_c,\mu_*\C)=\mu_*hom(\mu^*IC_c,\C)$ is supported on $S$. 
    
    Let $i_S:Q\setminus Q_S\rightarrow Q$ be the closed embedding of the complement of $Q_S$. We have a distinguished triangle 

     $$i_{S*}i_S^!hom(IC_c,IC_d)\rightarrow hom(IC_c,IC_d)\rightarrow j_{S*}j_S^*hom(IC_c,IC_d)\rightarrow$$
    
    As the image of $i_S$ does not intersect the support of $hom(IC_c,\mu_*\C)$ we get the result. 
\end{proof}

Thus, we can assume that all of our sheaves live only on $Q_S$. We will abuse the notation and will not write $j_S^*$ every time.

Now, we pass to the settings of $\overline{\Q}_l$ sheaves in the etale topology \cite{BBD}. Checking if the map $Comp_c$ of Proposition \ref{sheaf condition} is surjective can be done either in the analytic settings or in the etale setting (replacing the constant sheaf $\C$ with the constant sheaf $\overline{\Q}_l$) and the result does not depend on the setting. 

We may assume that $t$ is a diagonal matrix. The geometry of the orbits $\mathcal{E}$ on $Q$ depends only on whether ratios of values of $t$ are powers of $\sqrt{q_r}$. We work over finite field that contains $\sqrt{q_r}$. Let $\F_s$ be such a finite field and let $Fr\in Gal(\overline{\F_s}/\F_s)$ be the geometric Frobenius element. We recall the notions of purity and weights (see \cite{cataldo} Subsection 3.1 for more detail).

\begin{defn}
    Let $Z$ be an algebraic variety defined over $\F_s$. Let $\mathcal{F}$ be a constructible $\overline{\Q}_l$ sheaf on $Z$. We say that $\mathcal{F}$ is pure of weight $w$ if for any $x\in Z$ a point defined over $\F_{s^m}$. All the eigenvalues of $Fr^m$ on the stalk  $\mathcal{F}_x$ are algebraic with absolute value $s^\frac{mw}{2}$.

    We say the a sheaf is mixed of weight $\leq w$( $\geq w$) if it is filtered by pure sheaves whose eigenvalues have absolute values $\leq w$ (respectively $\geq w$).  

    Let $Z=\cup_\alpha Z_\alpha$ be stratified and let $i_\alpha:Z_\alpha\rightarrow Z$ be the embedding.
    
    We say that a complex of sheaves constant on each strata, $\mathcal{F}$ is $*$ pure of weight $w$ if $H^i(i^*_\alpha\mathcal{F})$ is pure of weight $i+w$ on $Z_\alpha$. We say that $\mathcal{F}$ is $!$ pure of weight $w$ if $H^i(i^!_\alpha\mathcal{F})$ is pure of weight $i+w$ on $Z_\alpha$. We say that $\mathcal{F}$ is pure if it is both $*$ pure and $!$ pure. 
\end{defn}

Gabber proved (see \cite{cataldo} Theorem 3.1.6) that the intersection cohomology sheaf of a connected variety of dimension $d$ is pure of weight $d$.

We use the same notation for orbits and varieties we have over $\C$ also over $\F_s$. From now on our $IC$ sheaves are $l$-adic sheaves, we also use the same notations for them.

First, we need a general result about pure sheaves.

\begin{lemma}\label{Exts are pure}
    \begin{enumerate}
        \item Let $\mathcal{F}$ be $*$ pure and let $\mathcal{G}$ be $!$ pure, then $hom(\mathcal{F},\mathcal{G})$ is $!$ pure.
        \item Let $\mathcal{F}$ be $!$ pure and let $i_\alpha:Z_\alpha\rightarrow Z$ be an open embedding of stratum. Then the map of cohomologies $H^\bullet(\mathcal{F})\rightarrow H^\bullet(i_\alpha^*\mathcal{F})$ is surjecitve.
    \end{enumerate}
\end{lemma}

\begin{proof}
    The first item follows from the identity $i_\beta^!hom(\mathcal{F},\mathcal{G})=hom(i_\beta^*\mathcal{F},i_\beta^!\mathcal{G})$ for any embedding  $i_\beta:Z_\beta\rightarrow Z$.
    The second item follows from Lemma 4.1.4 of \cite{ginzburg2025pointwisepurityderivedsatake}.
\end{proof}

\begin{prop}\label{orbits must intersect}
    Let $\omega$ be the dualizing sheaf on $S$.
    Let $\mathcal{F}\in D_c^b(Q_S)$ be a $*$ pure sheaf, then the natural map $\bigoplus_{d\in\mathcal{E}_S}Ext^{\bullet}(\mathcal{F},IC_d)\otimes Ext^\bullet(IC_d,\mu_*\omega)\rightarrow Ext^\bullet(\mathcal{F},\mu_*\omega)$ is surjective.
\end{prop}

\begin{proof}
    We prove this by induction on the support of $\mathcal{F}$. We denote by $supp_S(\mathcal{F})$ the subset of orbits in $\mathcal{E}_S$ which lie in the closure of the support of $\mathcal{F}$. We prove that in the statement of the proposition it is enough to take the sum over $supp_S(\mathcal{F})$.
    
    The basis of the induction is the case where $supp_S(\mathcal{F})$ is the closed orbit in $\mathcal{E}_S$ (it is closed in $Q_S$). Denote this closed orbit by $C_0$ and denote $S_0=C_0\cap S$. Let $i:C_0\rightarrow Q_S,i_S:S_0\rightarrow S$ be the closed embeddings. 

    In this case we claim that the map $Ext^{\bullet}(\mathcal{F},IC_{C_0})\otimes Ext^\bullet(IC_{C_0},\mu_*\omega)\rightarrow Ext^\bullet(\mathcal{F},\mu_*\omega)$ is surjective. As $C_0$ is a closed orbit we have $IC_{C_0}=i_*\overline{\Q}_l[dim(C_0)]$.

    We compute that $Ext^{\bullet}(\mathcal{F},IC_{C_0})=Ext^{\bullet}(\mathcal{F},i_*\overline{\Q}_l[dim(C_0)])=Ext^\bullet(i^*\mathcal{F},\overline{\Q}_l[dim(C_0)]))$.
    
    Also, $Ext^\bullet(IC_{C_0},\mu_*\omega)=Ext^\bullet(i_*\overline{\Q}_l[dim(C_0)],\mu_*\omega)=Ext^\bullet(\mu^*i_*\overline{\Q}_l[dim(C_0)],\omega)$, by proper base change this space is equal to $Ext^\bullet(i_{S*}\overline{\Q}_l[dim(C_0)],\omega)$. 
    
    The map $i_S$ is a closed embedding so we have $i_{S*}=i_{S!}$, therefore $Ext^\bullet(i_{S*}\overline{\Q}_l[dim(C_0)],\omega)$ is equal to $$Ext^\bullet(\overline{\Q}_l[dim(C_0)],i_S^!\omega)=Ext^\bullet(\overline{\Q}_l[dim(C_0)],\overline{\Q}_l[2dim(S_0)])=H^\bullet(S_0[2dim(S_0)-dim(C_0)])$$  
    
    Lastly, $Ext^\bullet(\mathcal{F},\mu_*\omega)=Ext^\bullet(\mu^*\mathcal{F},\omega)$.

    We know that $\mathcal{F}$ is constant on $C_0$ so we can write $i^*\mathcal{F}=\oplus_i\overline{\Q}_l^{\alpha_i}[i]$ for some integers $\alpha_i$. We also have $\mu^*\mathcal{F}=\oplus_i i_{S*}\overline{\Q}_l^{\alpha_i}[i]$ by our assumption on the support of $\mathcal{F}$.
    
    Then, $Ext^\bullet(\mu^*\mathcal{F},\omega)=Ext^\bullet(\oplus_i i_{S*}\overline{\Q}_l^{\alpha_i}[i],\omega)=Ext^\bullet(\oplus_i \overline{\Q}_l^{\alpha_i}[i],i^!_{S}\omega)=Ext^\bullet(\oplus_i \overline{\Q}_l^{\alpha_i}[i],\overline{\Q}_l[2dim(S_0)])$.

    We want to show that the following natural map is surjective.
$$Ext_{C_0}^\bullet(\oplus_i\overline{\Q}_l^{\alpha_i}[i],\overline{\Q}_l[dim(C_0)])\otimes H^\bullet(S_0[2dim(S_0)-dim(C_0)])\rightarrow Ext_{S_0}^\bullet(\oplus_i \overline{\Q}_l^{\alpha_i}[i],\overline{\Q}_l[2dim(S_0)])$$ 

    Another way to write the map is as the map $$\oplus_i \overline{\Q}_l^{\alpha_i}\otimes H^\bullet(C_0[dim(C_0)-i])\otimes H^\bullet(S_0[2dim(S_0)-dim(C_0)])\rightarrow \oplus_i \overline{\Q}_l^{\alpha_i}\otimes H^\bullet(S_0[2dim(S_0)-i])$$

    The dimension shifts cancel out and we have $$\oplus_i \overline{\Q}_l^{\alpha_i}\otimes H^\bullet(C_0[-i])\otimes H^\bullet(S_0)\rightarrow \oplus_i \overline{\Q}_l^{\alpha_i}\otimes H^\bullet(S_0[-i])$$ This map is clearly surjective. It is enough to take just a copy of $\overline{\Q}_l$ in $H^0(C_0)$ to get surjectivity.

    Next we move to the general case. 

    Choose some linear order on $\mathcal{E}_S$ that refines the order of closure containment. Let $C\in supp_S(F)$, assume that we already know the result for sheaves whose support contains only orbits smaller than $C$.

    Let $Q_C\subset Q_S$ be the union of all orbits whose closure contains $C$. Let $j:Q_C\rightarrow Q_S$ be the open embedding and let $i:Q_S\setminus Q_C\rightarrow Q_S$ be the closed embedding of the complement. 

    We have a distinguished triangle

    $$j_!j^*\mathcal{F}\rightarrow \mathcal{F}\rightarrow i_*i^*\mathcal{F}\rightarrow$$


    We also have the following distinguished triangle which gives a long exact sequence of Ext groups.

    $$hom(i_*i^*\mathcal{F},\mu_*\omega)\rightarrow hom(\mathcal{F},\mu_*\omega) \rightarrow hom(j_!j^*\mathcal{F},\mu_*\omega) \rightarrow$$


    Notice that the support of $i_*i^*\mathcal{F}$ is smaller than the support of $\mathcal{F}$. The map $i$ is a closed embedding, in particular it is proper and thus $i_*i^*\mathcal{F}$ is also a pure sheaf.  By induction we can apply the result of the proposition to this sheaf. We get that the following  map is surjective. 
    
    $$\bigoplus_{D<C}Ext^{\bullet}(i_*i^*\mathcal{F},IC_D)\otimes Ext^\bullet(IC_D,\mu_*\omega)\rightarrow Ext^\bullet(i_*i^*\mathcal{F},\mu_*\omega)$$ 


    The only orbit from $\mathcal{E}_S$ in the support of the sheaf $j^*\mathcal{F}$ is $C$. We can apply the same argument we used for the basis of the induction to get that the following map is surjective
    
    $$Ext^{\bullet}(j^*\mathcal{F},j^*IC_C)\otimes Ext^\bullet(j^*IC_C,j^*\mu_*\omega)\rightarrow Ext^\bullet(j^*\mathcal{F},j^*\mu_*\omega)$$ 

    Now we combine both parts. 

    By Lemma \ref{Exts are pure} we get that the maps $Ext^{\bullet}(\mathcal{F},IC_C)\rightarrow Ext^{\bullet}(j^*\mathcal{F},j^*IC_C)$ and
    
    $Ext^{\bullet}(\mathcal{F},\mu_*\omega)\rightarrow Ext^{\bullet}(j^*\mathcal{F},j^*\mu_*\omega)$ are surjective. We used the fact that $IC_C$ is pure, in particular ! pure. We also used the facts that $\mu^*\mathcal{F}$ is $*$ pure and that $\omega$ is the equal to the $IC$ sheaf of $S$ shifted and thus pure.

    Thus, for every element of $Ext^\bullet(\mathcal{F},\mu_*\omega)$ we can find an element of $Ext^{\bullet}(\mathcal{F},IC_C)\otimes Ext^{\bullet}(\mathcal{F},\mu_*\C)$ that has the same image in $Ext^\bullet(j_!j^*\mathcal{F},\mu_*\omega)$. By the long exact sequence of cohomology associated to 
    
    $$hom(i_*i^*\mathcal{F},\mu_*\omega)\rightarrow hom(\mathcal{F},\mu_*\omega) \rightarrow hom(j_!j^*\mathcal{F},\mu_*\omega) \rightarrow$$
    
    we get that that it is enough to show surjectivity on $Ext^\bullet(i_*i^*\mathcal{F},\mu_*\omega)$ which we already know.
\end{proof}

As an immediate corollary we get that.

\begin{cor}\label{MOS condition}
    Assume that for $c\in\mathcal{E}$, the following map is not surjective. $$\bigoplus_{k>0,d\in\mathcal{E}}Ext^k(IC_c,IC_d)\otimes \bigoplus Ext^{\bullet}(IC_d,\mu_*\C_{S^a})\rightarrow\bigoplus Ext^\bullet(IC_c,\mu_*\C_{S^a})$$ Then $c\in \mathcal{E}_S$.
\end{cor}

\begin{proof}
    Assume that $c\notin\mathcal{E}_S$, we know that $IC_c$ is pure and that $Ext^k(IC_c,IC_d)=0$ for every $d\in \mathcal{E}_S$ and $k\leq 0$. By Proposition \ref{orbits must intersect} we are done.
\end{proof}

\begin{theorem}\label{rep conjecture is true}
     Let $\pi$ be an irreducible representation of $G$ with $\pi^I\neq 0$. Let $(t,n)$ be the Deligne Langlands parameter of $\pi$. Let $a=(t,v)\in G^\vee\times\C^\times$. If $Z(\pi)^\vee$ is $X$ distinguished then $(M^\vee)^a\neq 0$ and $n\in (\mathfrak{g}^{*\vee})^a$ is in the image of the moment map $\mu:(M^\vee)^a\rightarrow (\mathfrak{g}^{*\vee})^a$. 
\end{theorem}

\begin{proof}
    The result $(M^\vee)^a\neq 0$ follows from Proposition \ref{non_empty_fixed_points} and Corollary \ref{central cahrachter}. The result about $n$ being in the image of the moment map follows from Proposition \ref{sheaf condition} and Corollary \ref{MOS condition}.
\end{proof}

\begin{Remark}
     Theorem \ref{rep conjecture is true} can be translated to a condition on the distinction of $\pi$ in terms of Zelevinsky parameters, i.e. multi-segments (see \cite{Zelevinsky1980}). The obtained condition is the same one proven in in \cite{MitraOffenSayag2017} (see Theorem 1.1). It states that the lengths of all segments in the multi-segment parameterizing $\pi$ are even. See \cite{MitraOffenSayag2017} for more detail. 

    This condition is not a sufficient condition.
\end{Remark}

\begin{prop}
    Let $C\in \mathcal{E}_S$, Denote by $Q_C\subset Q$ the union of all orbits whose closure contains $C$. Let $j:Q_C\rightarrow Q_S$ be the open embedding.
    
    The map $\bigoplus_{D\in\mathcal{E}_S}Ext^{\bullet>0}(IC_C,IC_D)\otimes Ext^\bullet(IC_D,\mu_*\omega)\rightarrow Ext^\bullet(IC_C,\mu_*\omega)$ is surjective if and only if the map $\bigoplus_{D\in\mathcal{E}_S,C\subset \overline{D},D\neq C}Ext^{\bullet}(j^*IC_C,j^*IC_D)\otimes Ext^\bullet(j^*IC_D,j^*\mu_*\omega)\rightarrow Ext^\bullet(j^*IC_C,j^*\mu_*\omega)$ is surjective. 
\end{prop}

\begin{proof}
    Let $i$ be the closed embedding which is the complement of $j$. We have the following distinguished triangle which gives a long exact sequence of $Ext$ groups.

    $$hom(i_*i^*IC_C,\mu_*\omega)\rightarrow hom(IC_C,\mu_*\omega) \rightarrow hom(j_!j^*IC_C,\mu_*\omega) \rightarrow$$

    Now, like in the proof of Proposition \ref{orbits must intersect} we get that $Ext^\bullet(i_*i^*IC_C,\mu_*\omega)$ can be obtained from orbits which are contained in the closure of $C$. These orbits also can not contribute to $Ext^\bullet(j_!j^*IC_C,\mu_*\omega)$ as if $C\nsubseteq \overline{D}$ and then $j^*IC_D=0$. Thus the only way to get the elements in $Ext^\bullet(IC_C,\mu_*\omega)$ whose image in $Ext^\bullet(j_!j^*IC_C,\mu_*\omega)$ is non zero is by using orbits from $Q_C$. Which means that we must have 
    elements in the image of the map $\bigoplus_{D\in\mathcal{E}_S,C\subset \overline{D},D\neq C}Ext^{\bullet}(j^*IC_C,j^*IC_D)\otimes Ext^\bullet(j^*IC_D,j^*\mu_*\omega)\rightarrow Ext^\bullet(j^*IC_C,j^*\mu_*\omega)$.

    We can lift elements of $Ext^{\bullet}(j^*IC_C,j^*IC_D)\otimes Ext^\bullet(j^*IC_D,j^*\mu_*\omega)$ to $Ext^{\bullet}(IC_C,IC_D)\otimes Ext^\bullet(IC_D,\mu_*\omega)$ using Lemma \ref{Exts are pure}.
\end{proof}

We immediately get the following:

\begin{cor}\label{maximal orbits}
    Let $c\in \mathcal{E}_S$ be a maximal orbit, let $\pi$ be the irreducible representation with a $I$ fixed vector attached to $c$. The representation $Z(\pi)^\vee$ is $X$ distinguished.
\end{cor}

\begin{Remark}
    The representations $Z(\pi)^\vee$ described by Corollary \ref{maximal orbits} are precisely the spherical irreducible representations, i.e. the representations with a $\mathbf{G}(\s)$ fixed vector. 
\end{Remark}

\begin{subsection}*{Example and explicit computations}

In this Subsection, we give an explicit description of the space $Q_S$ and the orbits $\mathcal{E}$. We compute some examples.

We relate $Q$ to quiver loci. 

Recall that we have $a=(t,v)$, we can assume that $t$ is diagonal of the form 

$t=diag(t_1v^{-1},...,t_nv^{-1},t_1v,...,t_nv)$. We can change $t$ by a central element and think about $t=diag(t_1,...,t_n,t_1q_r,...,t_nq_r)$. We can assume that each $t_i$ is a power of $q_r$. We define $(m_i)^{i=\infty}_{i=-\infty}$, $m_i$ is the number of times $q_r^i$ appears in $t$. We consider the equioriented quiver of type $A_\infty$ and we look at the space of representations with dimensions $m_i$. As $m_i$ is non zero only at finitely many places we actually have a finite number of equioriented quivers of type $A$ and a representation of each one. It is easy to see that $\mathfrak{g}^a$ is isomorphic to the space of such representations. The group $C_{G^\vee}(t)$ is isomorphic to $\Pi_{i}GL_{m_i}$. It acts on this spaces of quiver representations in the obvious way. As everything splits to a product over a finite number of quivers we can assume without the loss of generality that we have a single quiver. This is equivalent to all non zero elements of $m_i$ being consecutive.

\begin{Remark}
    In the case where there are entries of $t$ whose ratio is not power of $q_r$ we also get a product of representations of disjoint quivers.
\end{Remark}

We abuse the notation and denote by $Q$ the space of all quiver representations with dimensions $m_i$.

We need to explain how $S$ and its embedding $\mu$ can be seen in this picture. We can restrict $m_i$ to the places where it is non zero and by changing the indices assume that the dimensions of the representations are $m_0,...,m_l$. As $t=diag(t_1,...,t_n,t_1q_r,...,t_nq_r)$ we know that there are integers $k_0,...,k_{l-1}$ such that $m_0=k_0,m_l=k_{l-1}$ and $m_i=k_i+k_{i-1}$ for $1\leq i\leq l-1$. Each quiver representation is given by a set of matrices $A_i\in M_{m_{i+1},m_i}$ for $0\leq i\leq l-1$. The space $S\subset Q$ corresponds to matrices of the from $A_0=\begin{pmatrix}
    I \\B_0
\end{pmatrix}, A_{l-1}=\begin{pmatrix}
    -B_{l-2} & I
\end{pmatrix}$ and $A_i=\begin{pmatrix}
    -B_{i-1}& I\\C_i& B_i
\end{pmatrix}$ for $1\leq i\leq i-2$. There is no condition on $B_i\in M_{k_{i+1},k_i}$ and $C_i\in M_{k_{i+1},k_{i-1}}$. By abuse of notation we denote the subspace of such representations by $S$ and we have the embedding $\mu:S\subset Q$. We also denote the collection of $\Pi_{i}GL_{m_i}$ orbits in $Q$ by $\mathcal{E}$.

The problem we need to solve is the description of all orbits $c\in \mathcal{E}$ such that the map $\bigoplus_{d\in\mathcal{E}}Ext^{\bullet>0}(IC_c,IC_d)\otimes Ext^\bullet(IC_d,\mu_*\omega)\rightarrow Ext^\bullet(IC_c,\mu_*\omega)$ is not surjective. Such orbits will be called relevant orbits.

The orbits $\mathcal{E}$ can be parametrized by Zelevinsky parameters, i.e. multi-segments (see Definition \ref{multisegments}).



We give an example for which we can solve the problem completely. In this example we again work with sheaves over $\C$ in the analytic setting.

\begin{exmp}
Consider the following dimensions $m_0=m_1=m_2=m_3=m$ and the values  $k_0=m,k_1=0,k_2=m$. In this cases a quiver representation is given by three $m\times m$ matrices $A_0,A_1,A_2$. We have an action of the group $GL_m^4$.

The matrices in $S$ are of the form $A_0=I,A_1=C_0$ and $A_2=I$. 

Clearly, we only need to consider orbits $c\in\mathcal{E}$ whose closure intersect $S$. For such orbits $A_0,A_2$ are invertible matrices and so the $GL_m^4$ orbit is determined by the rank of $A_1$. Thus, there are $m+1$ orbits we should consider. Denote by $C_r, IC_r$ the the orbit of rank $r$ matrices and its $IC$ sheaf respectively. 

Denote by $rank_r$ the subspace of $M_{m,m}$ of $m\times m$ matrices of rank exactly $r$.

$C_r$ is a subset of $M_{m,m}\times M_{m,m}\times M_{m,m}$. The condition is only on the second matrix so $C_r=M_{m,m}\times rank_{r}\times M_{m,m}$. The map $\mu:M_{m,m}\rightarrow M_{m,m}\times M_{m,m}\times M_{m,m}$ is given by $\mu(A)=(I,A,I)$. We have $IC_r=\C_{M_{m,m}}\boxtimes IC_{rank_r}\boxtimes\C_{M_{m,m}}[2m^2]$ and $\mu^*IC_r=IC_{rank_r}[2m^2]$.

We also have $Ext^\bullet(IC_r,IC_k)=Ext^\bullet(IC_{rank_r},IC_{rank_k})$.

Thus, instead of $\bigoplus_{D\in\mathcal{E}}Ext^{\bullet>0}(IC_C,IC_D)\otimes Ext^\bullet(IC_D,\mu_*\C)\rightarrow Ext^\bullet(IC_C,\mu_*\C)$ we have $\bigoplus_{k}Ext^{\bullet>0}(IC_{rank_r},IC_{rank_k})\otimes Ext^\bullet(IC_{rank_k}[2m^2],\C)\rightarrow Ext^\bullet(IC_{rank_r}[2m^2],\C)$. We can cancel out the shift by $[2m^2]$ and get the map $$\bigoplus_{k}Ext^{\bullet>0}(IC_{rank_r},IC_{rank_k})\otimes Ext^\bullet(IC_{rank_k},\C)\rightarrow Ext^\bullet(IC_{rank_r},\C)$$

We claim that this map is surjective unless $r=m$. For $k=m$ we $IC_{rank_k}=\C[m^2]$ and $Ext^{\bullet}(IC_{rank_r},IC_{rank_k})=Ext^{\bullet}(IC_{rank_k},IC_{rank_r})=Ext^{\bullet}(\C[m^2],IC_{rank_r})=H^{\bullet-m^2}(IC_{rank_r})$. We also have $Ext^\bullet(IC_{rank_k},\C)=Ext^\bullet(\C[m^2],\C)=H^{\bullet-m^2}(\C)$. 

We also have  $Ext^\bullet(IC_{rank_r},\C)=Ext^\bullet(\C[2m^2],IC_{rank_r})=H^{\bullet-2m^2}(IC_{rank_r})$. Overall we have the map  $H^{\bullet-m^2}(IC_{rank_r})\otimes H^{\bullet-m^2}(\C)\rightarrow H^{\bullet-2m^2}(IC_{rank_r})$. We have the condition of $\bullet>0$ in the term $H^{\bullet-m^2}(IC_{rank_r})$. If $r<m$ there is nothing in cohomological degree $-m^2$ so we have a surjective map. Thus, the only values of $r$ for which we do not have a surjective map is $r=m$.  

This implies that the only distinguished representation with this specific central character has Zalevisnky parameter $[0,1,2,3]\times m$. 
\end{exmp}

\end{subsection}

\end{section}

\newpage

\renewcommand{\C}{\mathcal{C}}

\appendix

\section{Appendix A, Boundary degenerations}\label{A1}

In this Appendix, we discuss the general theory of boundary degenerations of symmetric spaces. We will describe the action of the affine Hecke algebra on the space of $I$ invariant compactly supported functions on a most degenerate boundary degenerations. We will prove the results about boundary degenerations we used to prove Proposition \ref{ideal generaotrs}.

In this Appendix, we use the notion of boundary degeneration introduced for symmetric varieties by Delorme in \cite{NeighborhoodsInf} and for spherical varieties by Sakellaridis in \cite{sakellaridis2017periods}. Additional relevant results are discussed in \cite{Delorme2014PaleyWienerTF}. We almost follow the notations of \cite{NeighborhoodsInf}, in our notation the action of $G$ on $X$ is a left action.

\begin{subsection}{Delorme's results on boundary degenerations}

In this Subsection we recall the definitions and results of \cite{NeighborhoodsInf} that are used in this appendix.

Let $G$ be any reductive connected group split over $F$, let $\sigma:G\rightarrow G$ be an algebraic involution. Let $H=G^\sigma$ and let $X=G/H$. 

\begin{defn}
    Let $P$ be a parabolic subgroup of $G$, we call $P$ a $\sigma$ parabolic if $P$ and $\sigma(P)$ are opposites, i.e. $P\cap \sigma(P)$ is a Levi subgroup of $P$.
\end{defn}

\begin{defn}
    Let $T\subset G$ be a torus, we say that $T$ is $\sigma$ split if it is split over $F$ and $\sigma(t)=t^{-1}$ for any $t\in T$.
\end{defn}

\begin{defn}
    Let $P=MU$ be a $\sigma$ parabolic with its Levy decomposition, let $A_M$ be the maximal $\sigma$ split torus in the center of $M$.
    Let $\Sigma(P)$ be the set of $A_M$ roots in the Lie algebra of $P$. Let $\Delta(P)$ be the set of simple roots in $\Sigma(P)$. Let $|\cdot|$ be the norm of $F$, denote
    $A^+_P=\{a\in A_M||\alpha(a)|\geq 1, \alpha\in \Delta(P)\}$ and  $A^{++}_P=\{a\in A_M||\alpha(a)|> 1, \alpha\in \Delta(P)\}$.
\end{defn}

\begin{defn}
    Let $P_\emptyset=M_\emptyset U_\emptyset$ be a minimal $\sigma$ parabolic. Let $A_\emptyset$ be a maximal $\sigma$ stable torus contained in the center of $M_\emptyset$. Let $P=MU$ be a $\sigma$ parabolic that contains $P_\emptyset$. Let $C>0$ and let $\Delta(U,A_\emptyset)$ be the simple roots of $A_\emptyset$ in the Lie algebra of $U$, let $A^+_\emptyset(P,C)=\{a\in A_\emptyset ||\alpha(a)|\geq C,\alpha\in \Delta(U,A_\emptyset)\}$.
\end{defn}

\begin{defn}
    Let $P=MU$ be a $\sigma$ parabolic, let $X^{Lev}_M=\{g\in G| g^{-1}A_Mg \text{ is a $\sigma$ split torus}\}$. Denote $X_M=X^{Lev}_M/H$ and $X_P=G\times^{P^-} X_M$, $P^-=\sigma(P)$ being the parabolic opposite to $P$.
    $X_P$ is called the boundary degeneration of $X$ with respect to $P$.
\end{defn}

\begin{Remark}
$M/H\cap M$ is a connected component of $X_M$ (see Lemma 1 of \cite{NeighborhoodsInf} for a proof) and $G/(H\cap M)U^-$ is a connected component of $X_P$ (see Lemma 2 of \cite{NeighborhoodsInf} for a proof).
\end{Remark}

\begin{defn}
    Let $P=MU$ be a $\sigma$ parabolic, let $a_n\in A_M$ we say that $a_n\rightarrow_P \infty$ if for every $\alpha\in \Sigma(U,A_M)$, which are the roots of $A_M$ in the Lie algebra of $U$, we have $|\alpha(a_n)|\rightarrow\infty$.
\end{defn}

\begin{Remark}
    In this appendix we mostly consider the case of $P$ being a minimal $\sigma$ parabolic. In this case we denote $a_n\rightarrow_P \infty$ simply by $a_n\rightarrow\infty$.
\end{Remark}

\begin{Remark}
    Notice that if $t\in A_P^{++}$ then $t^n\rightarrow_P\infty$ 
\end{Remark}

We will the need to following lemma. 

\begin{lemma}{(Lemma 7 of \cite{NeighborhoodsInf})}
\label{lemma7}
Let $P=MU$ be a $\sigma$ parabolic, let $a_n\in A_M$ such that $a_n\rightarrow_P\infty$. Let $g_n\in G$ be a sequence converging to $g$, $g_n\rightarrow g$ such that for all $n\in\N$ $g_na_nH=a_nH$ then $g\in (H\cap M)U^-$.
\end{lemma}

\begin{defn}
    Let $A_i$ be representatives of the maximal $\sigma$ split tori of $G$ (under the action of $H$ by conjugation). For each $i$ choose $x_i$ such that $A_i=x^{-1}_i A_\emptyset x_i$ with one of them being $1$. 

    Denote by $W_R(A_\emptyset)=N_R(A_\emptyset)/Z_R(A_\emptyset)$ for any group $R$ that contains $A_\emptyset$.
    
    Let $W_i$ be a set of representatives in $N_G(A_\emptyset)$
    of $W_G(A_\emptyset)/W_{x^{-1}_iHx_i}(A_\emptyset)$. Let $W^G_{M_\emptyset}=\cup W_ix^{-1}_i$ and $\chi^G_{M_\emptyset}=\{xH|x^{-1}\in W^G_{M_\emptyset}\}$.

    Now let $P$ be a $\sigma$ parabolic and $C>0$, define $N_X(P,C)=\cup_{x\in \chi^G_{M_\emptyset}}A_\emptyset^+(P,C)xH$
\end{defn}

We can now formulate a key result concerning the existence of Bernstein maps $$e:S(X_P) \rightarrow S(X)$$  from a degeneration $X_P$ of a $p$-adic symmetric spaces.

\begin{theorem}{(Theorem 3 of \cite{NeighborhoodsInf})}
\label{theorem3}
    Let $P$ be a $\sigma$ parabolic, there exists a map $e:S(X_P)^I\rightarrow S(X)^I$ which is $G$ equivariant and for any compact set $\Omega\subset G$ there exists $C$ large enough such that we have for $x\in \Omega N_X(P,C)$, e sends the characteristic function of $Ix(M\cap H)U^-$ to the characteristic function of $IxH$.
\end{theorem}

\end{subsection}

\begin{subsection}{The action of $H(G,I)$ on $S(X_\emptyset)^I$}

Let $P$ be a minimal $\sigma$ parabolic. Let $T\subset P$ be a maximal split torus. Let $U\subset P$ be the unipotent radical of $P$ and let $U^-=\sigma(U)$ be the unipotent radical of the opposite parabolic $P^-$. Let $M=P\cap \sigma(P)$, we have $P=MU$ and $P^-=MU^-$. We denote $X_P$ be $X_\emptyset$.

In this subsection we describe the action of $H(G,I)$ on $S(X_\emptyset)^I$

We need the following proposition:

\begin{prop}\label{normalizing}[Proposition 4.7 in \cite{Helminck}]
The torus $T$ normalizes $H_P=U^-(M\cap H).$ In fact, the $\sigma$ split part of $T$ is contained in the center of $P\cap \sigma(P).$
\end{prop}

Thus, there is a right action of $T$ on $S(X_\emptyset)^I$ which commutes with the action of $H(G,I)$.

Let $\Tilde{\Delta}$ be the set of simple reflections inside $W_{aff}$. For any $s\in \Tilde{\Delta}$ we have $T_s\in H(G,I)$.

Let $w\in W_{aff}$, the next proposition describes the action of $T_s$ on a characteristic function of the form $1_{IwH_P}\in S(X_\emptyset)^I$. 

We use the language of the Bruhat Tits building of $G$, as was done in \cite{my}. Let $\B$ be the extended Bruhat Tits building of $G$. Let $\Omega$ be the fundamental group of $G$. The set of $I$ orbits on $X$ can be identified with $H$ orbits of $\Omega$ colored chambers in $\B$. 

\begin{prop}\label{action}
For an element $w\in W_{aff}$ exactly one of the following three statements holds:
\begin{enumerate}
    \item There exists $t_n\rightarrow\infty$ such that $T_s1_{Iwt_nH}=1_{Iswt_nH}$
    \item There exists $t_n\rightarrow\infty$ such that $T_s1_{Iwt_nH}=q1_{Iswt_nH}+(q-1)1_{Iwt_nH}$
    \item There exists $t_n\rightarrow\infty$ such that $T_s1_{Iwt_nH}=q1_{Iswt_nH}$
\end{enumerate}
And we have respectively
\begin{enumerate}
    \item $T_s1_{IwH_P}=1_{IswH_P}$
    \item $T_s1_{IwH_P}=q1_{IswH_P}+(q-1)1_{IwH_P}$
    \item $T_s1_{IwH_P}=q1_{IswH_P}$
\end{enumerate}
\end{prop}

\begin{proof}
We begin with proving that at least one of the first 3 options occurs. 

Let $\A\subset \B$ be the $\sigma$ stable apartment corresponding to the torus $T$.

Under the correspondence between $I$ orbits on $X$ and $H$ orbits on $\Omega$ colored chambers in $\B$, the orbit  $IwH$ for $w\in W_{aff}$ corresponds to a $H$ orbit of a  chamber in $\A$. We choose such a chamber and denote it by $\C$. Let $o\in \Omega$ be such that $IwH$ corresponds to the $H$ orbit of $(\C,o)$.

Let $f$ be a facet of $\C$ of codimension 1, such that the action of $s$ on $H(\C,o)$ is by reflection across $f$.

By Lemma 5.5 of \cite{my} if $T_s1_{IwH}\notin \{1_{IswH},q1_{IswH}+(q-1)1_{IwH},q1_{IswH}\}$ then $f$ is equal to the intersection of two $\sigma$ stable apartments that are not in the same $H$ orbit.

The torus $T$ is a maximal $\sigma$ stable split torus inside of a minimal $\sigma$ parabolic. As such its $\sigma$ split part $T^-=\{t\in T|\sigma(t)=t^{-1}\}$ is of maximal dimension among all $\sigma$ split parts of $\sigma$ stable tori (see Proposition 4.7 of \cite{Helminck}). Thus, by Proposition 7.11 of \cite{my}, the involution $\sigma$ switches the half spaces on both sides of $f$. Therefore, the affine hyper plane generated by $f$ must contain the $\sigma$ fixed part of $\A$. In particular, no two faces $f$ like that can be parallel and thus there are only finitely many faces $f$ like that. Therefore, for any chamber in $\A$ there is a direction such that moving far enough in this direction keeps the chamber away from faces of the mentioned form. Let $t\in T$ be an element acting on $\A$ by such a translation. Write $t=t^-t^+$ with $t^+\in T\cap H$ and $t^-\in T^-$. By Proposition \ref{normalizing} we have $t^-\in A_M$. As our condition on $t$ is an open one, we may assume that for every root $\alpha$ of $A_M$ in the lie algebra of $U$, $\alpha(t^-)\neq 1$. We can find $t_n$, a sub sequence of the powers of $t^-$ such that one of the first options occurs.

It is clear that at most one of the last three options occurs, so it is enough to prove the correspondence between them.

We show the correspondence for the first case, the others are similar.
We assume that there exists $t_n\rightarrow\infty$ such that $T_s1_{IwtH}=1_{IswtH}$.
We can write $T_s1_{IwH_P}=\sum a_i1_{Ig_iH_P}$ for some $a_i\in\mathbb{C}$ and $g_i\in G$.

By Theorem \ref{theorem3} for $n$ large enough we have $e(1_{Ig_it_nH_P})=1_{Ig_it_nH}$. By Proposition \ref{normalizing} we have for any $n$, $T_s1_{Iwt_nH_P}=\sum a_i1_{Ig_it_nH_P}$. For $n$ large enough we have $1_{Iswt_nH}=T_s1_{Iwt_nH}=e(T_s1_{Iwt_nH_P})=\sum a_ie(1_{Ig_it_nH_P})=\sum a_i1_{Ig_it_nH}$. Thus, for every $i$ we have $Ig_it_nH=Iswt_nH$. By Lemma \ref{lemma7}  we get $Ig_iH_P=IswH_P$ for every $i$ and the result follows.
\end{proof}

\begin{cor}
     $G=IW_{aff}H_P$
\end{cor}

\begin{proof}
    We need to show that every $I$ orbit on $X_P$ can be represented by an element of $W_{aff}$. Consider the $H_P$ orbits on the set of $\Omega$ colored chambers in $\B$. Assuming that there is an $I$ orbit on $X_P$ not represented by $W_{aff}$ there is a chamber in $\B$ that is not in the $H_P$ orbits of the chambers inside $\A$. As the chambers of $\B$ with the adjacency relation form a connected graph, we can find two chambers $\C_1,\C_2$ which intersect on a face and such that $\C_1\subset \A$ and $\C_2$ is not $H_P$ conjugate to a chamber in $\A$. Let $s\in\Tilde{\Delta}$ be the simple reflection that corresponds to the the reflection across the face $\C_1\cap\C_2$ of $\C_1$ and let $w\in W_{aff}$ such that $IwH_P$ corresponds to the $H_P$ orbit of $(\C_1,o)$ for some $o\in \Omega$. The function $T_s1_{IwH_P}$ is non zero on the $H_P$ orbit of $(\C_2,o)$, but by the previous proposition it is non zero only on $I$ orbits which are represented by $W_{aff}$.
\end{proof}
\end{subsection}

\subsection{The map $e:S(X_\emptyset)^I \to S(X)^I$}

In this Subsection we make two assumptions about $X$. Our main example in this paper $X=GL_{2n}/Sp_{2n}$ satisfies both assumptions. 

\begin{enumerate}
    \item The algebraic group $\mathbf{G}$ is defined over $\s$ and the maximal compact subgroup $\mathbf{G}(\s)$ is $\sigma$ stable.
    \item The space $X$ has minimal rank, i.e. that $rank(X)=rank(G)-rank(H)$. 
 \end{enumerate}

We begin with the following simple observation.
\begin{prop}
\label{onto}
The Bernstein map $e:S(X_\emptyset)^I\rightarrow S(X)^I$ is surjective.
\end{prop}

\begin{proof}
Any characteristic function of a chamber in $\A$ generates the space of functions supported on chambers in $\A$. By Theorem \ref{theorem3} there is a characteristic function of an $I$ orbits which is mapped to a characteristic function of an $I$ orbit.

The space $X$ is of minimal rank so $H$ acts transitively on $\sigma$ stable apartment (see Proposition 2.1 of \cite{Ressayre_2010}). Thus the $H$ orbits of chambers in $\A$ are equal to the $H$ orbits of all chambers in $\B$.
\end{proof}

\begin{prop}\label{prop:matching}
There exists $w\in W$ such that $e(1_{IwH_P})=1_{IwH}$.
\end{prop}
\begin{proof}
 We call an element $w\in W_{aff}$ good if 
$$e(1_{IwH_P})=1_{IwH}$$
By Theorem \ref{theorem3} a dominant element $t\in T$ which is dominant enough is good.

Let $t\in T$ be a dominant good element. We can write $t=t_+t_-$ with $t_+\in T\cap H=T\cap H_P$ and $t_-\in T^-$, then $t_-$ is also good. 

Let $0\in \A$ be the point corresponding to the maximal compact subgroup $\mathbf{G}(\s)$. By assumption $0$ is $\sigma$ stable. 

Let $\C_0\subset \A$ be a chamber that contains $0$ such that for $t\in T$ dominant, $\C_0$ is in the direction of $t\sigma(t)^{-1}$. The $H$ orbit of $(\C_0,1)$ corresponds to $IwH$ some $w\in W$, we claim that this $w$ is good.

Recall the length function defined on chambers by $l_\sigma(\C)=d(\C,\sigma(\C))$. Here, $d$ is the distance function between chambers in $\B$. We also have the usual length function $l$ on $W_{aff}$.

We claim that $l_\sigma(t_-\C_0)=2l(t_-)+l_\sigma(\C_0)$. This is equivalent to $d(t^2_-\C_0,\sigma(\C_0))=l(t^2_-)+d(\C_0,\sigma(\C_0))$. For this it is enough to prove that there is a gallery between $t^2_-\C$ and $\sigma(\C)$ which contains $\C_0$. We know that $t^2_-=t\sigma(t)^{-1}$. The line passing through $0$ in the direction defined by $t$ intersects $\C_0$ and the line passing through $0$ in the direction defined by $\sigma(t)$ intersects $\sigma(\C_0)$. The direction defined by $t^2_-$ is pointing from $\sigma(\C_0)$ to $\C_0$ which implies that there is a minimal gallery between $\sigma(\C_0)$ and $t^2_-\C_0$ that contains $\C_0$.

Thus, we can find a minimal gallery between $t_-\C_0$ and $\sigma(t_-\C_0)$ by combining a minimal galleries from $t_-\C_0$ to $\C_0$, from $\C_0$ to $\sigma(\C_0)$ and from $\sigma(\C_0)$ to $\sigma(t_-\C_0)$.

Let $\C_i=w_i\C_0$ be a chamber in a minimal gallery from $t_-\C_0$ to $\C_0$, we claim that $w_iw$ is good. Proving this will finish the proof.

We prove that $w_i$ is good by induction. 
Assume that $w_{i-1}$ is good, there is a simple  reflection $s$ such that $w_i=sw_{i-1}$. By Proposition \ref{action} it is enough to prove that there are $t_n\in T$ such that $t_n\rightarrow \infty$ and $l_\sigma(sw_it_n)>l_\sigma(w_it_n)$. Let $t_n=t^n$, it is enough to show that $l_\sigma(sw_it^n_-)>l_\sigma(w_it^n_-)$.

This follows from the fact that translating a minimal gallery between $\C_0$ and $t_-\C_0$ by $t_-$ gives a minimal gallery between $t_-\C_0$ and $t^2_-\C_0$. Thus we can find a minimal gallery from $sw_it^n_-\C_0$ to $\C_0$ be concatenating the minimal gallery from $sw_it^n_-\C_0$ to $t^n_-\C_0$ and a gallery from $t^n_-\C_0$ to $\C_0$. Such a minimal gallery passes through $w_it^n_-\C_0$. We can use this to construct a minimal gallery from $sw_it^n_-\C_0$ to $\C_0$, to $\sigma(\C_0)$, to $\sigma(sw_it^n_-\C_0)$. Such a gallery passes through $w_it^n_-\C_0$ and  $\sigma(w_it^n_-\C_0)$.
\end{proof}

\bibliographystyle{alphaurl}
\bibliography{mybib}

\end{document}